\documentclass{article}
\usepackage{amsmath, calligra, amssymb,amsthm,authblk,mathrsfs,mathtools}
\usepackage{xypic}
\usepackage{hyperref}

\newif\ifdraft\draftfalse
\newif\ifcite\citefalse
\newif\ifblow\blowtrue
\ifdraft\usepackage{datetime}\else\relax\fi
\ifdraft\ifcite\usepackage{showkeys}\else\usepackage[notcite,notref]{showkeys}\fi\fi

\DeclareMathOperator\res{\mathrm{Res}}
\DeclareMathOperator\Dbc{\mathit{D}^{\mathit{b}}_\mathit{c}}
\DeclareMathOperator\Perv{\mathrm{Perv}}
\DeclareMathOperator\Shv{\mathrm{Shv}}
\DeclareMathOperator\pH{\prescript{p}{}{\mathit{H}}}
\DeclareMathOperator\supp{\mathrm{supp}}

\DeclareMathOperator\diag{\mathrm{diag}}
\DeclareMathOperator\asc{asc}
\DeclareMathOperator\ch{ch}
\DeclareMathOperator\cycletype{cycletype}
\DeclareMathOperator\des{des}

\DeclareMathOperator\pr{\mathrm{pr}}
\newcommand\red{\mathrm{red}}
\newcommand\CF{\mathrm{CF}}
\newcommand{\bm}{{\mathbf{m}}}
\newcommand{\da}{{\mathord{\downarrow}}}
\newcommand{\ua}{{\mathord{\uparrow}}}
\newcommand{\sH}{{\mathscr{H}}}
\newcommand{\sHrs}{\sH^{\mathrm{rs}}}

\newcommand{\sX}{{\mathscr{X}}}
\newcommand{\sq}{{q}}
\newtheorem{conjecture}[equation]{Conjecture}
\newtheorem{corollary}[equation]{Corollary}
\newtheorem{proposition}[equation]{Proposition}
\newtheorem{theorem}[equation]{Theorem}
\newtheorem{fact}[equation]{Fact}

\newtheorem{lemma}[equation]{Lemma}

\theoremstyle{definition}
\newtheorem{definition}[equation]{Definition}

\theoremstyle{remark}
\newtheorem{remark}[equation]{Remark}

\newtheorem{example}[equation]{Example}

\DeclareMathOperator{\sHom}{\mathscr{H}\kern -2pt\mathit{om}}
\DeclareMathOperator{\sEnd}{\mathscr{E}\kern -1pt\mathit{nd}}
\DeclareMathOperator{\End}{\mathrm{End}}
\DeclareMathOperator{\Hom}{\mathrm{Hom}}
\DeclareMathOperator{\Aut}{\mathrm{Aut}}
\DeclareMathOperator{\GL}{\mathbf{GL}}
\DeclareMathOperator{\Gr}{\mathrm{Gr}}
\DeclareMathOperator{\Ad}{\mathrm{Ad}}
\DeclareMathOperator{\rH}{\mathrm{H}}

\DeclareMathOperator\id{\mathrm{id}}
\DeclareMathOperator\IC{\mathrm{IC}}

\DeclareMathOperator\colim{\mathrm{colim}}
\DeclareMathOperator\HH{\mathrm{H}}

\newcommand{\car}{\mathbf{car}}
\newcommand{\carss}{\car^{\mathrm{rs}}}
\newcommand\arr{\ifinner \to\else\longrightarrow\fi}
\newcommand\arrto{\ifinner\mapsto\else\longmapsto\fi}

\newcommand{\gr}{\mathfrak{g}^{\mathrm{r}}}
\newcommand{\grs}{\mathfrak{g}^{\mathrm{rs}}}
\newcommand{\dr}{\mathbf{d}_{\mathrm{r}}}
\newcommand*\Bell{\ensuremath{\boldsymbol\ell}}

\begin{document}

\hfuzz=3pt
\title{Unit Interval Orders and the Dot Action on the Cohomology of Regular Semisimple Hessenberg Varieties }
\author[1]{Patrick Brosnan\thanks{Department of Mathematics, 1301 Mathematics Building, University of Maryland, College Park, MD 20742, USA; \texttt{pbrosnan@umd.edu}.  Partially supported by NSF grant DMS-1361159.}}

\author[2]{Timothy Y. Chow\thanks{Center for Communications Research, 805 Bunn Drive, Princeton, NJ 08540, USA; \texttt{tchow@alum.mit.edu}.}}

\affil[1]{University of Maryland, College Park}
\affil[2]{Center for Communications Research, Princeton}
\maketitle

\begin{abstract}
  Motivated by a 1993 conjecture of Stanley and Stembridge, Shareshian
  and Wachs conjectured that the characteristic map
  takes the
  character of the
  dot action of the symmetric group on the cohomology
  of a regular semisimple Hessenberg variety to $\omega X_G(t)$, where
  $X_G(t)$ is the chromatic quasisymmetric function of the
  incomparability graph~$G$ of the corresponding natural unit interval
  order, and $\omega$ is the usual involution on symmetric functions.
  We prove the Shareshian--Wachs conjecture.

  Our proof uses the local invariant cycle theorem of
  Beilinson--Bernstein--Deligne to obtain a surjection, which we call
  the local invariant cycle map, from the cohomology of a regular
  Hessenberg variety of Jordan type~$\lambda$ to a space of local
  invariant cycles.  As $\lambda$ ranges over all partitions, the
  local invariant cycles collectively contain all the information
  about the dot action on a regular semisimple Hessenberg variety.  We
  then prove a result showing that, under suitable hypotheses,
  the local invariant cycle map is an isomorphism if and only if the
  special fiber has palindromic cohomology.  (This is a general theorem, which
  independent of the Hessenberg variety context.) 
  Applying this result to the universal family of Hessenberg varieties, we 
  show that, in our case, the surjections are actually isomorphisms,
  thus reducing the Shareshian--Wachs conjecture to computing the
  cohomology of a regular Hessenberg variety.
  But this cohomology has already been described combinatorially by
  Tymoczko, and, using a new reciprocity theorem for certain quasisymmetric
  functions, we show that Tymoczko's  description coincides
  with the combinatorics of the chromatic quasisymmetric function.
  \end{abstract}
\ifdraft\footnote{\today:\currenttime}\else\relax\fi

\section{Introduction}\label{intro}

Let $G$ be the incomparability graph of a unit interval order
(also known as an \emph{indifference graph}),
i.e., a finite graph whose vertices are closed unit intervals
on the real line, and whose edges join overlapping unit intervals.
It is a longstanding conjecture \cite{stanley-stembridge}
related to various deep conjectures about immanants
that if $G$ is such a graph,
then the so-called \emph{chromatic symmetric function}
$X_G$ studied by Stanley~\cite{stanley-xg} is $e$-positive,
i.e., a nonnegative combination of elementary symmetric functions.
(In fact, Stanley and Stembridge conjectured something seemingly
more general, but Guay-Paquet~\cite{guay-paquet} has reduced
their conjecture to the one stated here.)
Early on, Haiman~\cite{haiman} proved that the expansion
of~$X_G$ in terms of Schur functions
has nonnegative coefficients, and Gasharov~\cite{gasharov}
showed that these coefficients enumerate certain combinatorial
objects known as \emph{$P$-tableaux}.
It is well known that if $\chi$ is a character of
the symmetric group~$S_n$, then the image of~$\chi$
under the so-called characteristic map~$\ch$
\begin{equation}
\label{eq:ch}
\ch \chi := \frac{1}{n!} \sum_{\sigma\in S_n} \chi(\sigma) \,
    p_{\cycletype(\sigma)}
\end{equation}
(where $p$ here denotes the power-sum symmetric function)
is a nonnegative linear combination of Schur functions,
with the coefficients giving the multiplicities of
the corresponding irreducible characters of~$S_n$.
One may therefore suspect that $X_G$ is the image under~$\ch$
of the character of some naturally occurring representation of~$S_n$,
but until recently, there was no candidate, even conjecturally,
for such a representation.

Meanwhile, independently and seemingly unrelatedly,
De~Mari, Procesi, and Shayman~\cite{demari-procesi-shayman}
inaugurated the study of \emph{Hessenberg varieties}.
Let $\bm = (m_1, m_2, \ldots, m_{n-1})$ be a weakly
increasing sequence
of positive integers satisfying $i\le m_i\le n$ for all~$i$,
and let $s:\mathbb{C}^n \to \mathbb{C}^n$ be a linear transformation.
The (type~A) Hessenberg variety $\sH(\bm,s)$ is defined by
\begin{equation}
\label{eq:hessenberg}
\sH(\bm,s) := \{
\mbox{complete flags $F_0 \subseteq F_1 \subseteq \cdots \subseteq F_{n} :
sF_i \subseteq F_{m_i}$ for $1\leq i<n$}\}.
\end{equation}
The geometry of a Hessenberg variety depends on the Jordan form of~$s$.
If the Jordan blocks have distinct eigenvalues then we say that $s$ is \emph{regular}, and, by
extension, we also say that  
$\sH(\bm, s)$ is \emph{regular}.  
Similarly, if $s$ is diagonalizable then we say that 
$\sH(\bm, s)$ is \emph{semisimple}.  We say that $s$ has \emph{Jordan type $\lambda$}
if $\lambda$ is the partition of $n$ given by the sizes of the Jordan blocks on $s$. 
Hessenberg varieties have many interesting properties,
but of particular interest to us is the fact that
there is a representation, called the \emph{dot action}, of~$S_n$
on the cohomology of
regular semisimple Hessenberg varieties.  This dot action was first defined by Tymoczko, who
asked for a complete description of it~\cite{tymoczko};
e.g.,  a combinatorial formula for
the multiplicities of the irreducible representations
and/or for the character values.

A connection between these two apparently unrelated topics has
been conjectured by
Shareshian and Wachs~\cite{shareshian-wachs-0,shareshian-wachs}.
Motivated by the $e$-positivity conjecture,
they have generalized $X_G$ to something they call the
\emph{chromatic quasisymmetric function} $X_G(t)$ of a graph,
which is a polynomial in~$t$ with power series coefficients
that reduces to~$X_G$ when $t=1$.
They also noted that, if we are given a sequence~$\bm$ as above,
and we let $G(\bm)$ be the undirected
graph on the vertex set $\{1, 2, \ldots, n\}$
such that $i$ and~$j$ are adjacent if $i<j \le m_i$,
then $G(\bm)$ is an indifference graph,
and moreover that every indifference graph
is isomorphic to some $G(\bm)$.
They then made the following conjecture.
Let $\omega$ denote the usual involution 
on symmetric functions~\cite[Section 7.6]{stanley-ec2}.
\begin{conjecture}
\label{conj:sw}
Let $y$ be a regular semi-simple $n\times n$-matrix and let
$\chi_{\bm,d}$ denote the character of the dot action
on $\HH^{2d}(\sH(\bm, y))$.
Then $\ch \chi_{\bm,d}$ equals
the coefficient of $t^d$ in
$\omega X_{G(\bm)}(t)$.
\end{conjecture}
This conjecture is intriguing because not only   would it
answer Tymoczko's question, but it would also open up the possibility
of proving the $e$-positivity conjecture by geometric techniques.

The main result of the present paper is a proof of
Conjecture~\ref{conj:sw} (Theorem~\ref{thm:sw}).
The linchpin of our proof is the following result
(which is stated more formally later as Theorem~\ref{dsf}).

\begin{theorem}
\label{thm:linchpin}
Let $\lambda=(\lambda_1,\ldots, \lambda_{\ell})$ be a partition of~$n$.
Let $s$ be a regular element with Jordan type~$\lambda$,
and let $S_\lambda := S_{\lambda_1} \times \cdots\times S_{\lambda_\ell}$
be the corresponding Young subgroup of the symmetric group~$S_n$.
Consider the restriction of $\chi_{\bm,d}$ to $S_\lambda$.
Then the dimension of the subspace fixed by~$S_\lambda$
equals the Betti number $\beta_{2d}$ of $\sH(\bm,s)$.
\end{theorem}

What Theorem~\ref{thm:linchpin} does is to reduce the problem of
computing the dot action on a regular semisimple Hessenberg variety
to computing the cohomology of regular (but not necessarily semisimple)
Hessenberg varieties.
Fortunately, this latter task has already been largely carried out
by Tymoczko~\cite{tymoczko-linear},
who has given a combinatorial description of the Betti numbers $\beta_{2d}$
for all Hessenberg varieties in type~A.
So with Theorem~\ref{thm:linchpin} in hand,
all that remains to prove Conjecture~\ref{conj:sw} is
to give a bijection between Tymoczko's combinatorial description
and the combinatorics of $\omega X_{G(\bm)}(t)$.
More precisely, let $m_{\lambda}$ denote the monomial symmetric function
associated to the partition $\lambda$. (See~\cite[Section~7.3]{stanley-ec2}
or \S\ref{sym} below for monomial symmetric functions.)  It is then a
standard fact, proved explicitly in Proposition~\ref{prop:frobenius}
below,  
that the dimension of the subspace fixed by $S_\lambda$
in a representation~$\chi$ is the coefficient of
$m_\lambda$
in the monomial symmetric function expansion of $\ch \chi$.
So the first step of our proof is to compute the
coefficients $c_{d,\lambda}(\bm)$ of $t^d m_\lambda$
in the monomial symmetric function expansion
of $\omega X_{G(\bm)}(t)$.
We do this with a generalization of
a combinatorial reciprocity theorem of Chow (Theorem~\ref{thm:reciprocity}).
This yields a description of $c_{d,\lambda}(\bm)$ that is almost,
but not quite, identical to Tymoczko's description of $\beta_{2d}$;
we show that the descriptions are equivalent
by describing an explicit bijection
between the two (Theorem~\ref{thm:betti}).
As a corollary (Corollary~\ref{pbetti}), we derive the fact that the
Betti numbers of regular Hessenberg varieties
form a palindromic sequence
(even though the varieties are not smooth),
because Shareshian and Wachs have proved that
$\omega X_{G(\bm)}(t)$ is palindromic.

The idea behind the proof of Theorem~\ref{thm:linchpin} is to show
that Tymoczko's dot action coincides with the monodromy action for the
family $\sH^{\mathrm{rs}}(\bm)\to\mathfrak{g}^{\mathrm{rs}}$ of Hessenberg varieties
over the space of regular semisimple $n\times n$ matrices
(Theorem~\ref{tdot}).  This allows us to apply results from the theory
of local systems and perverse sheaves to questions involving the dot
action.  In particular, the local invariant cycle theorem of
Beilinson--Bernstein--Deligne, which is stated in our context as
Theorem~\ref{thm:bbd}, implies that there is a surjective map from
the cohomology of a regular Hessenberg variety to the space of local
invariants of the monodromy action near a regular element $s$ in the
space $\mathfrak{g}$ of all $n\times n$-matrices.

In Theorem~\ref{t1}, we show that the local invariant cycle map is an
isomorphism if and only if the Betti numbers of the special fiber are
palindromic in a suitable sense.  This is a general result in that it
holds for any projective morphism of smooth, complex, quasi-projective
schemes.  Then, in Theorem~\ref{tiy}, we show that the local invariant
cycles near a regular element $s$ with Jordan type~$\lambda$ coincide
with the $S_{\lambda}$ invariants of the dot action on the regular
semisimple Hessenberg variety.  The latter fact is proved by a
monodromy argument that uses the Kostant section.

Here is a brief description of the contents by
section. Section~\ref{prelim} mainly fixes notation and gives
preliminary results.  Section~\ref{reciprocity} proves the
combinatorial reciprocity theorem, Theorem~\ref{thm:reciprocity},
mentioned above.  Section~\ref{sbetti} proves Theorem~\ref{thm:betti}
on the Betti numbers of regular Hessenberg varieties, and derives
palindromicity as a corollary of a theorem of Shareshian and Wachs.
Section~\ref{lmlfg} reviews the concept of local monodromy and the
related notion of a good fundamental system of neighborhoods to a
point in topological space.  Section~\ref{sbn} proves Theorem~\ref{t1}
on palindromicity and  the local invariant cycle map.  Along the way we review the proof
of the local invariant cycle map from~\cite{bbd} and prove a slightly
stronger version of it (Theorem~\ref{big-lic}) using the Kashiwara conjecture~\cite{KashiwaraConj} (which, by the work
of several authors, is now a theorem).  We also prove Theorem~\ref{t-pal-main},  a more general
version of our theorem on palindromicity and the local invariant
cycle map.  
Section~\ref{gc} proves Proposition~\ref{pgal} on the local monodromy
of a Galois cover, which is applied later (in Lemma~\ref{myoung}) to
compute the local monodromy near a matrix $s$ of type~$\lambda$.
Section~\ref{sghs} introduces the family $\sH(\bm)\to\mathfrak{g}$ of
Hessenberg varieties.  Finally, Section~\ref{mtd} shows that the
monodromy action coincides with Tymoczko's dot action, and uses this
fact to prove Theorem~\ref{thm:sw}, which is a restatement of
Conjecture~\ref{conj:sw}.

\subsection{Previous work}
Prior to our work, Conjecture~\ref{conj:sw} was already known 
for some graphs~$G$: 
a complete graph (trivial),
a complete graph minus an edge~\cite{teffthesis},
a complete graph minus a path of length three (Tymoczko, unpublished),
and a path (by piecing together known results as
explained in~\cite{shareshian-wachs}).
In a different direction,
Abe, Harada, Horiguchi and Masuda (AHHM) 
proved that the multiplicity of the trivial representation is indeed as
predicted by Conjecture~\ref{conj:sw}.  Hearing about this development
and reading the last paragraph of the research announcement~\cite{ahhm}, which explains how to
compute the multiplicity of the trivial representation in terms of the
regular nilpotent Hessenberg variety, partially inspired our own proof.
Full details of the work of AHHM appeared  on the arXiv in~\cite{ahhm2}
shortly after the first draft~\cite{BrosnanChow} of this paper.
(Later, Abe, Horiguchi and Masuda
computed the ring structure on regular semisimple
Hessenberg varieties of type $(m_1,n,\ldots, n)$ in a way that is
compatible with the dot action and used this to deduce the structure
of the cohomology as an $S_n$-representation in that case~\cite{ahm}.)

In addition to the above work, very shortly after posting the first
version of this paper on the arXiv, we learned of the series of papers
by Chen, Vilonen and Xue (CVX) studying the the motives of certain
generalized Hessenberg varieties as well as the action of monodromy as
they vary in families.  (See, for example, \cite{cvxh}.)  The context
of this work is different from ours because, roughly speaking,
generalized Hessenberg varieties are much further from combinatorics
than the Hessenberg varieties which appear in our work.  However,
to the best of our knowledge, CVX were the first to exploit
the idea of studying a universal family of Hessenberg varieties using its
monodromy.

\subsection{Later work} Since the first version of this paper appeared
on the arXiv there have been a few related developments that may be
helpful for understanding this work. Firstly, Guay-Paquet posted a proof of
Conjecture~\ref{conj:sw} which is completely independent of and, in
many ways, complementary to our proof~\cite{gp-second}. 

On the other hand, using her generalization~\cite{precup} of
Tymoczko's computation of the Betti numbers of Hessenberg varieties in
type A, Precup generalized our palindromicity results
(Corollaries~\ref{pbetti} and~\ref{dimcor}) to regular Hessenberg
varieties for arbitrary complex semi-simple Lie
algebras~\cite{PrecupPal}.  As it happens,
our proof of palindromicity is
rather indirect, relying in an essential way on the
chromatic quasisymmetric function and a palindromicity theorem for
that function proved by Shareshian and Wachs (\cite[Corollary
4.6]{shareshian-wachs}).  So, even in the type A case dealt with in
this paper, Precup's direct proof is an important contribution.

Finally, using some of the ideas in this paper,  Harada and Precup have proved the $e$-positivity of the coefficients of $X_{G(\mathbf{m})}(t)$ for certain
sequences $\mathbf{m}$ corresponding to abelian ideals in the Lie algebra of strictly upper-triangular matrices~\cite{HaradaPrecup}.  This generalizes
Remark 4.4 to Theorem 4.3 of~\cite{stanley-stembridge}.

\subsection{Acknowledgments} 
Brosnan would like to thank Jim Carrell for discussions about
Hessenberg varieties and equivariant cohomology, Mark Goresky for
mentioning Springer's Bourbaki article~\cite{springer}, which eventually 
convinced us to look for a description of Tymoczko's dot action in
terms of monodromy, and Joseph Bernstein for a conversation in MPI in Bonn that helped convince us to reformulate the palindromicity theorem, Theorem~\ref{t1},
in terms of perverse sheaves on the base as we do in Theorem~\ref{jbpal}.
He would also like to thank Naj\-muddin Fakhruddin 
and Nero Budur for advice on 
\S\ref{lmlfg}-\ref{sbn},  Hiraku Abe for discussions over email about flatness
which led to the formulation of Theorem~\ref{t.flat} below, and Tom
Haines for discussions about Grothendieck's simultaneous resolution.
He thanks the
 Institute for Advanced Study for hospitality in Princeton, NJ during
 the Fall of 2014 when this work started.

We both thank the referees for many thoughtful suggestions and corrections.

\section{Preliminaries}\label{prelim}
We fix some notation that will be used throughout the paper.

\subsection{General notation}\label{gn}

We let $\mathbb{P}$ denote the positive integers.
If $n\in \mathbb{P}$, we let $[n]$ denote the set $\{1, 2,\ldots, n\}$.

The vector $\bm = (m_1, \ldots, m_{n-1})$
will always denote a \emph{Hessenberg function},
by which we mean a sequence of positive integers satisfying
\begin{enumerate}
\item $m_1\le m_2\le \cdots \le m_{n-1} \le n$, and
\item $m_i \ge i$ for all $i$.
\end{enumerate}
We also define
\begin{equation}
\label{eq:mdim}
\left|\mathbf{m}\right| := \sum_{i=1}^{n-1} (m_i-i).
\end{equation}
Given $\bm$, let $P(\bm)$ denote the poset
on the vertex set $[n]$ whose order relation $\prec$ is given by
\[ i\prec j \Longleftrightarrow j \in \{m_i+1, m_i+2, \ldots, n\}.\]
Such a poset is called a \emph{natural unit interval order}.
The \emph{incomparability graph} $G(\bm)$ is the undirected
graph on the vertex set $[n]$ in which $i$ and~$j$ are adjacent
if and only if $i$ and~$j$ are incomparable in~$P(\bm)$.
In other words, if $i<j$ then $i$ and~$j$ are adjacent in $G(\bm)$
if and only if $j\le m_i$.

An \emph{integer partition}
$\lambda = (\lambda_1, \lambda_2, \ldots, \lambda_\ell)$ of
a positive integer~$n$ is
a weakly decreasing sequence of
positive integers that sum to~$n$.
Each $\lambda_i$ is a \emph{part} of~$\lambda$,
and the number of parts of~$\lambda$ is denoted by $\ell(\lambda)$.
The \emph{Young diagram} of~$\lambda$
comprises $\ell$ rows of boxes, left-justified,
with $\lambda_i$ boxes in the $i$th row from the top.
We write  $\lambda\vdash n$ to indicate that $\lambda$ is a partition
of $n$.  

A \emph{composition}
$\alpha = (\alpha_1, \alpha_2, \ldots, \alpha_\ell)$ of
a positive integer~$n$ is a (not necessarily monotonic) sequence 
of positive integers that sum to~$n$.
Each $\alpha_i$ is a \emph{part} of~$\alpha$,
and the number of parts of~$\alpha$ is denoted by $\ell(\alpha)$.
It can be useful to visualize a composition of~$n$
by drawing vertical bars in some subset of the $n-1$ spaces between 
consecutive objects in a horizontal line of $n$ objects;
the parts are then the numbers of objects between successive bars.
Motivated by the equivalence between compositions and sets of bars,
we define:
\begin{itemize}
\item $|\alpha|$ for the number of bars of~$\alpha$
	(equivalently, $|\alpha| = \ell(\alpha)-1$;
	\emph{CAUTION:} $|\alpha|$ is \emph{not}
	the sum of the parts of~$\alpha$);
\item $\overline \alpha$ for the composition that has bars in precisely
the positions where $\alpha$ does \emph{not} have bars;
\item $\alpha \cup \beta$ for the composition whose bars comprise
the union of the bars of $\alpha$ and the bars of~$\beta$; and
\item $\alpha \le \beta$ if the bars of~$\alpha$
are a subset of the bars of~$\beta$.
\end{itemize}

We write $S_n$ for the symmetric group.
If $S_n$ acts in the usual way on a set of size~$n$,
and $\alpha$ is a composition of~$n$,
then the \emph{Young subgroup}~$S_\alpha$
is the subgroup
\begin{equation}
\label{eq:young}
S_{\alpha_1} \times S_{\alpha_2} \times \cdots \times S_{\alpha_\ell}
  \subseteq S_n
\end{equation}
comprising all the permutations that permute the first $\alpha_1$
elements among themselves, the next $\alpha_2$ elements among themselves,
and so on.

An \emph{ordered (set) partition}
$\sigma = (\sigma_1, \sigma_2, \ldots, \sigma_\ell)$ of
a finite set~$S$ is a sequence of pairwise disjoint non-empty subsets of~$S$
whose union is~$S$.

A \emph{sequencing} $\sq$ of a finite set~$S$ of cardinality~$n$
is a bijective map $\sq:[n] \to S$.
It is helpful to think of $\sq$ as the sequence $\sq(1), \ldots, \sq(n)$
of elements of~$S$.

By a \emph{digraph} we mean a finite directed graph with no loops or
multiple edges but that may have bidirected edges, i.e., it
may contain both $u\to v$ and $v\to u$ simultaneously.
If $D$ is a digraph, we write $\overline D$ for the
\emph{complement} of~$D$, i.e., the digraph with the same vertex set
as~$D$ but with a directed edge $u\to v$ if and only if
there does \emph{not} exist a directed edge $u\to v$ in~$D$.

\subsection{Symmetric and quasisymmetric functions}\label{sym}
We mostly follow the notation of Stanley~\cite{stanley-ec2}
for symmetric and quasisymmetric  functions.
For convenience, we recall some of the notation here.

Let $\mathbf{x} = \{x_1, x_2, x_3, \ldots\}$ be a countable set
of independent indeterminates.
If $\kappa:[n] \to \mathbb{P}$ is a map
then we write $\mathbf{x}_\kappa$ for the monomial
$x_{\kappa(1)} x_{\kappa(2)} \cdots x_{\kappa(n)}$.
A formal power series in $\mathbb{Q}[[\mathbf{x}]]=\mathbb{Q}[[x_1,x_2,\ldots ]]$ is a \emph{symmetric function}
if it is
of bounded degree and 
invariant under any permutation of the variables~$\mathbf{x}$.
We write $\Lambda$ for the subring of $\mathbb{Q}[[\mathbf{x}]]$
consisting of symmetric functions.   Then $\Lambda=\oplus_{n\geq 0} \Lambda_n$
where $\Lambda_n$ denotes the space of homogeneous symmetric functions of
degree $n$.

If $\lambda=(\lambda_1, \lambda_2, \ldots, \lambda_\ell)$ is
an integer partition,
then the \emph{monomial symmetric function} $m_\lambda$
is the symmetric function of minimal support that contains the monomial
$x_1^{\lambda_1}x_2^{\lambda_2}\cdots x_\ell^{\lambda_\ell}$.  For example,
\[ m_{2,1,1} = x_1^2 x_2 x_3 + x_2^2 x_1 x_3 + x_3^2 x_1 x_2 + 
               x_1^2 x_3 x_4 + x_3^2 x_1 x_4 + x_4^2 x_1 x_3 +  \cdots\]
There is an unfortunate conflict between
our notation for monomial symmetric functions and
our notation~$\bm$ for Hessenberg functions.
It should be clear from context which is meant since
the subscript of a monomial symmetric function is a partition,
whereas the entries of~$\bm$ have integer subscripts.

Set $h_n:=\sum_{\lambda\vdash n} m_{\lambda}$. Then, if $\lambda$ is a partition, set $h_{\lambda}=\prod h_{\lambda_i}$.
The $h_{\lambda}$ are called the
\emph{complete homogeneous symmetric functions}~\cite{stanley-ec2}.

Both $\{h_{\lambda}\}_{\lambda\vdash n}$ and $\{m_{\lambda}\}_{\lambda\vdash n}$ form bases of $\Lambda_n$.  So
we get a non-degenerate scalar product on $\Lambda_n$ (and on $\Lambda$ as well) by setting
\begin{equation} 
  \label{e.ScalarProd}
  \langle m_{\lambda},h_{\mu}\rangle =\delta_{\lambda\mu}\text{  (Kronecker delta)}
\end{equation}
as in \cite[Equation 7.30]{stanley-ec2}.
This scalar product is symmetric~\cite[Proposition 7.9.1]{stanley-ec2}.

Write $\CF_n$ for the space of $\mathbb{Q}$-valued class functions on $S_n$,
and set $\CF=\oplus_{n\geq 0} \CF_n$.  The \emph{characteristic map}~$\ch:\CF_n\to \Lambda_n$
is a function that sends class functions~$\chi$
on the symmetric group to symmetric functions via the formula
\begin{equation}
\ch \chi := \frac{1}{n!} \sum_{\sigma\in S_n} \chi(\sigma) \,
    p_{\cycletype(\sigma)}
\end{equation}
where $\cycletype(\sigma)$ is the integer partition consisting
of the cycle sizes of~$\sigma$, listed with multiplicity in
weakly decreasing order,
and $p$ denotes the power-sum symmetric function.
It turns out that $\ch$ is an isomorphism of $\mathbb{Q}$-vector spaces.
Moreover, if we give $\CF_n$ the standard inner product $\langle \cdot,\cdot\rangle$ on class functions, then $\ch$ is an isometry~\cite[Proposition 7.18.1]{stanley-ec2}:
\begin{equation}
  \label{e.iso}
  \langle \ch f, \ch g\rangle = \langle f, g\rangle.
\end{equation}

Note that all complex characters of finite dimensional representations
of $S_n$ are actually rational.  In fact, even the representations themselves
are realizable over $\mathbb{Q}$~\cite[Example 1, page 103]{SerreRep}.
So it makes sense to work with the $\mathbb{Q}$-valued class
functions even if we are interested in complex representations of $S_n$. 

As we explained in the introduction,
the following standard fact is an important ingredient in our proof.

\begin{proposition}
\label{prop:frobenius}
Let $\rho$ be a finite-dimensional complex representation of~$S_n$,
and let~$\chi$ be its character.  Let
$\ch\chi = \sum_\lambda c_\lambda m_\lambda$
be the monomial symmetric function expansion of $\ch\chi$.
Then $c_\lambda$ equals the dimension of the subspace
fixed by any Young subgroup $S_\lambda\subseteq S_n$.
In particular, knowing $c_\lambda$ for all $\lambda$
uniquely determines~$\chi$.
\end{proposition}

\begin{proof}
Let $\chi\da^{S_n}_{S_\lambda}$ denote the
restriction of~$\chi$ to~$S_\lambda$, and
let $d_\lambda$ be
the dimension of the subspace fixed by~$S_\lambda$.
Then $d_\lambda$ equals
the multiplicity of the trivial representation~$\mathbf{1}$
in $\chi\da^{S_n}_{S_\lambda}$, i.e.,
$d_\lambda = \langle \mathbf{1}, \chi\da^{S_n}_{S_\lambda}\rangle$.
By Frobenius reciprocity~\cite[Theorem 1.12.6]{sagan},
\begin{equation}
\langle \mathbf{1}, \chi\da^{S_n}_{S_\lambda}\rangle
= \langle \mathbf{1}\ua^{S_n}_{S_\lambda}, \chi\rangle,
\end{equation}
where $\mathbf{1}\ua^{S_n}_{S_\lambda}$ is the induction
of $\mathbf{1}$ from~$S_\lambda$ up to~$S_n$.
But $\ch \mathbf{1}\ua^{S_n}_{S_\lambda}$ is just
the  homogeneous symmetric function $h_\lambda$
\cite[Corollary 7.18.3]{stanley-ec2}.
The monomial symmetric functions and the complete
homogeneous symmetric functions are dual bases,
so $d_\lambda = \langle h_\lambda, \ch\chi\rangle = c_\lambda$.
\end{proof}

Let $\alpha = (\alpha_1, \alpha_2, \ldots, \alpha_\ell)$
be a composition of~$n$.  The \emph{monomial quasisymmetric function}
$M_\alpha$ is the formal power series defined by
\begin{equation}
\label{eqn:Malpha}
M_\alpha := \sum_{i_1<\cdots<i_\ell} x_{i_1}^{\alpha_1} \cdots 
   x_{i_\ell}^{\alpha_\ell},
\end{equation}
where the sum is over all strictly increasing sequences
$(i_1, \ldots, i_\ell)$ of positive integers.
In addition, we define a degree-zero monomial
quasisymmetric function by $M_\varnothing := 1$.
A formal power series is a \emph{quasisymmetric function}
if it is a finite
rational linear combination of monomial quasisymmetric functions.
We write $\mathscr{Q}$ for the algebra of quasisymmetric functions
and $\mathscr{Q}_n$ for the space of homogeneous quasisymmetric functions
of degree $n$~\cite[Section 7.19]{stanley-ec2}.
Clearly, we have $\mathscr{Q}=\oplus \mathscr{Q}_n$, and clearly $\Lambda$ is a subalgebra of $\mathscr{Q}$.  
Note that it is a proper subalgebra.  (For example, $M_{2,1}\in\mathscr{Q}\setminus \Lambda$.)

The \emph{fundamental quasisymmetric function} $F_\alpha$
of Gessel~\cite{gessel} is defined by
\begin{equation}
\label{eq:fundamental}
F_\alpha := \sum_{\beta\ge\alpha} M_\beta,
\end{equation}
and again we set $F_\varnothing := 1$.
By inclusion-exclusion,
\begin{equation}
\label{eq:incexc}
M_\alpha = \sum_{\beta\ge\alpha} (-1)^{|\beta|-|\alpha|} F_\beta.
\end{equation}

\subsection{Hessenberg varieties}

As mentioned in the introduction,
if $\bm$ is a Hessenberg function and
$s:\mathbb{C}^n \to \mathbb{C}^n$ is a linear transformation,
then we define the \emph{Hessenberg variety} (of type~A,
which is the only type that we consider in this paper) by
\begin{equation*}
\sH(\bm,s) := \{
\mbox{complete flags $F_0 \subseteq F_1 \subseteq \cdots \subseteq F_n :
sF_i \subseteq F_{m_i}$ for $1\leq i<n$}\}.
\end{equation*}
If the Jordan blocks of~$s$ have distinct eigenvalues then we say that
$\sH(\bm, s)$ is \emph{regular}, if $s$ is diagonalizable then we say
that $\sH(\bm, s)$ is \emph{semisimple}, and if $s$ is nilpotent then
we say that $\sH(\bm, s)$ is \emph{nilpotent}.  Since $\sH(\bm,s)$ can
equal $\sH(\bm,s')$ for $s\neq s'$ (e.g., if $s'-s$ is a constant),
this is a (very minor) abuse of terminology.   
We adopt the convention of writing
$y$ for~$s$ in the regular semisimple case.

\begin{remark}
  The Hessenberg varieties are defined on affine open subsets of the
  complete flag variety by fairly obvious equations.  So they are
  closed subschemes of the complete flag variety in a natural way.  In
  general, they are not irreducible.  For example, the regular
  semisimple Hessenberg variety corresponding to the function
  $\Bell=(1,2,\ldots, n-1)$ is a collection of $n!$ distinct points.
  They are also not always reduced.  For example, when $n=2$ and
  $\bm=\Bell$ as above, the regular nilpotent Hessenberg variety is
  defined by the equation $x^2=0$ in $\mathbb{A}^1$.
  (See~\cite[Theorem 7.6]{dt} for a much more general statement.)
  Hartshorne defines an abstract variety to be an integral separated
  scheme of finite type over an algebraically closed
  field~\cite[p.105]{hartshorne}.  So, perhaps, it is unfortunate that
  Hessenberg varieties are called varieties as they are not, in
  general, integral.  However, it happens that we are only interested
  in the Betti cohomology of these varieties in this paper.  So the
  non-reduced structure will not play a role.  Moreover, we will
  reserve the term ``Hessenberg \emph{scheme}'' for the families
  discussed in \S\ref{sghs}.   
  So we will stick with tradition and
  continue to call the schemes $\sH(\bm,s)$ Hessenberg
  \emph{varieties}.

However, we ask the reader to regard ``Hessenberg variety'' as one word.
The term \emph{variety} by itself will still refer to an integral, separated
scheme of finite type over an algebraically closed field.  
\end{remark}

\section{The chromatic quasisymmetric function}\label{reciprocity}

Given a graph $G$ whose vertex set is a subset of~$\mathbb{P}$,
Shareshian and Wachs~\cite{shareshian-wachs} define the
\emph{chromatic quasisymmetric function} $X_G(\mathbf{x},t)$ of~$G$.

\begin{definition}\label{cqsf}
Let $G$ be a graph whose vertex set~$V$ is a finite subset of~$\mathbb{P}$.
Let $C(G)$ denote the set of all proper colorings of~$G$,
i.e., the set of all maps $\kappa:V\to\mathbb{P}$ such
that adjacent vertices are always mapped to distinct positive integers.
Then
\begin{equation}
\label{eqn:XGt}
X_G(\mathbf{x}, t) := \sum_{\kappa\in C(G)}
t^{\asc\kappa}\,\mathbf{x}_\kappa,
\end{equation}
where
\[ \asc{\kappa} := \left|\left\lbrace \mbox{$\{u,v\}$} 
 : \mbox{$\{u,v\}$ is an edge of $G$ and $u<v$ and
   $\kappa(u) < \kappa(v)$}\right\rbrace \right|.
\]
\end{definition}

As Shareshian
and Wachs point out, it is obvious that $X_G(\mathbf{x},t)\in \mathscr{Q}[t]$.
On the other hand, while the proof of the following result, \cite[Theorem 4.5]{shareshian-wachs}, is not long,
the result itself is not at all obvious: 

\begin{theorem}[Shareshian--Wachs]\label{swsym}
  Suppose $G=G(\mathbf{m})$ is  the  incomparability graph of a natural
  unit interval order.  Then $X_G(\mathbf{x},t)\in\Lambda[t]$. 
\end{theorem}

For brevity, we sometimes write $X_G(t)$ for $X_G(\mathbf{x}, t)$.
It will be convenient for us to restate the definition of~$X_G(t)$
in terms of monomial quasisymmetric functions.  

\begin{proposition}
Let $G$ be a graph whose vertex set~$V$ is a finite subset of~$\mathbb{P}$.
Then
\begin{equation}
X_G(\mathbf{x},t) = \sum_{\sigma=(\sigma_1, \ldots, \sigma_\ell)}
t^{\asc\sigma}\, M_{|\sigma_1|, \ldots, |\sigma_\ell|},
\end{equation}
where the sum is over all ordered partitions~$\sigma$ of~$V$
such that every~$\sigma_i$ is a stable set of~$G$
(i.e., there is no edge between any two vertices of~$\sigma_i$),
and $\asc{\sigma}$ is the number of edges $\{u,v\}$ of~$G$ such
that $u<v$ and $v$ appears in
a later part of~$\sigma$ than $u$ does.
\end{proposition}

\begin{proof}
Given a coloring $\kappa\in C(G)$, let $\sigma_i$ be the set of vertices
that are assigned the $i$th smallest color.
Then it is immediate that
\begin{enumerate}
\item $\sigma=(\sigma_1, \ldots, \sigma_\ell)$ is an ordered partition
of the vertex set of~$G$;
\item $\sigma_i$ is a stable set for all~$i$; and
\item $\asc \kappa = \asc\sigma$.
\end{enumerate}
It is easy to see that if we sum $\mathbf{x}_\kappa$ over all
$\kappa\in C(G)$ that yield the same ordered partition~$\sigma$,
then we obtain the monomial quasisymmetric function $M_\alpha$
where the $i$th part~$\alpha_i$ of the composition~$\alpha$
is the cardinality $|\sigma_i|$ of~$\sigma_i$.
The proposition follows.
\end{proof}

We remark that if we set $t=1$ then the chromatic quasisymmetric
function specializes to the \emph{chromatic symmetric function}~$X_G$
of Stanley~\cite{stanley-xg}.

\subsection{Reciprocity}

If $f$ is a symmetric function, then a ``reciprocity theorem,''
loosely speaking, is a result that gives a combinatorial interpretation
of $\omega f$, where $\omega$ is a well-known involution
on symmetric functions~\cite[Section 7.6]{stanley-ec2}.
Since Conjecture~\ref{conj:sw} concerns $\omega X_G(t)$
rather than $X_G(t)$ itself, one might expect a reciprocity theorem
to be relevant.  This is indeed the case.
Specifically, the coefficients of the monomial symmetric function
expansion of $\omega X_G(t)$ play an important role in our arguments,
so we now introduce some notation for them.

\begin{definition}
Given a Hessenberg function~$\bm$,
we let $c_{d,\lambda}(\bm)$ be the coefficients defined by
the following expansion of $\omega X_{G(\bm)}(\mathbf{x}, t)$
in terms of monomial symmetric functions:
\begin{equation}
\label{eq:monom}
   \omega X_{G(\bm)}(\mathbf{x}, t) =
      \sum_d t^d \sum_\lambda c_{d,\lambda}(\bm) m_\lambda.
\end{equation}
\end{definition}

It is possible to derive a combinatorial interpretation for
$c_{d,\lambda}(\bm)$ by using 
the reciprocity theorem of Shareshian and Wachs
\cite[Theorem 3.1]{shareshian-wachs}.
However, as we now explain, we shall take a different route.

Our starting point is the observation that
Chow~\cite[Theorem 1]{chow} has proved a reciprocity theorem for
a symmetric function invariant of a digraph called the
\emph{path-cycle symmetric function}~$\Xi_D$.
There is a certain precise sense in which $\Xi_D$
is equivalent to Stanley's $X_G$ in the case of posets,
but the nice thing about reciprocity for $\Xi_D$
is that it naturally yields a combinatorial interpretation for the
coefficients of the monomial symmetric function expansion
of~$\omega \Xi_D$, which is not immediately evident from
Stanley's reciprocity theorem~\cite[Theorem 4.2]{stanley-xg} for~$X_G$.
This fact suggests the following plan: Generalize $\Xi_D$ to $\Xi_D(t)$
(just as Shareshian and Wachs have generalized $X_G$ to $X_G(t)$),
prove reciprocity for $\Xi_D(t)$,
and read off the desired combinatorial interpretation
of~$c_{d,\lambda}(\bm)$.
This plan works, and we now show how to carry it out.

We define the \emph{path quasisymmetric function} $\Xi_D(\mathbf{x},t)$
of a digraph~$D$; as its name suggests,
it enumerates paths only and not cycles
(since for our present purposes we do not care about enumerating cycles),
and it has a definition analogous to that of the chromatic quasisymmetric
function.

\begin{definition}
Let $D$ be a digraph whose vertex set~$V$ is a subset of~$\mathbb{P}$.
An \emph{ordered path cover} of~$D$ is an ordered pair $(\sq,\beta)$
such that $\sq$ is a sequencing of~$V$,
$\beta = (\beta_1, \ldots, \beta_\ell)$
is a composition of $n := |V|$, and 
\[ \sq(\beta_{i-1}+1)\to \sq(\beta_{i-1}+2)\to \cdots\to \sq(\beta_i)\]
is a directed path in~$D$ for all $i\in [\ell]$
(adopting the convention that $\beta_0=0$).
Define
\begin{equation}
\label{eq:xidt}
\Xi_D(\mathbf{x},t) := \sum_{(\sq,\beta)} t^{\asc \sq}\, M_\beta
\end{equation}
where the sum is over all ordered path covers $(\sq,\beta)$ of~$D$
and $\asc \sq$ is the number of pairs $\{u,v\}$ of vertices of~$D$
such that
\begin{enumerate}
\item either $u\to v$ and $v\to u$ are both edges of~$D$ or neither one is,
\item $u<v$, and 
\item $v$ appears later in the sequencing~$\sq$ than $u$ does.
\end{enumerate}
\end{definition}

For brevity, we sometimes write $\Xi_D(t)$ for $\Xi_D(\mathbf{x},t)$.
The chromatic quasi\-symmetric function and the path quasisymmetric function
coincide for posets.  More precisely, we have the
following proposition.

\begin{proposition}
\label{prop:xequalsxi}
Let $P$ be a
poset whose vertex set~$V$ is a finite subset of~$\mathbb{P}$.
Let $D(P)$ be the digraph on~$V$
that has an edge $u\to v$ if and only if $v\prec u$ in~$P$.
Let $G(P)$ be the incomparability graph of~$P$.
Then $\Xi_{D(P)}(\mathbf{x},t) = X_{G(P)}(\mathbf{x}, t)$.
\end{proposition}

\begin{proof}[Proof (sketch)]
The proof is mostly a routine verification
that the two definitions coincide in this special case.
Only a few points require some attention.
First, if $S$ is a stable subset in~$G(P)$,
then $S$ is a totally ordered subset of~$P$,
and hence there is exactly one directed path in~$D(P)$
through the vertices of~$S$.
Hence ordered partitions~$\sigma$ of~$V$
such that every $\sigma_i$ is a stable set of~$G(P)$
are in bijective correspondence with
ordered path covers $(\sq,\alpha)$ of~$D(P)$.
Second, because $P$ is a poset, it is not possible
for $u\prec v$ and $v\prec u$ simultaneously,
so the condition that ``either $u\to v$ and $v\to u$
are both edges of~$D(P)$
or neither one is'' is equivalent to adjacency in~$G(P)$.
Third, one might worry that $\asc\sq$ counts some pairs
$\{u,v\}$ where $v$ appears later in the \emph{sequencing}
but in the same \emph{path} while $\asc\sigma$ counts only
pairs from different parts, but in fact this cannot happen
because vertices in the same path are part of the same
totally ordered subset of~$P$ and thus have a directed edge
between them in exactly one direction.
\end{proof}

Although we are ultimately interested in expansions in terms of
monomial \emph{symmetric} functions,
it turns out that the proofs are more naturally stated in terms of
monomial \emph{quasisymmetric} functions.
So we need to describe the action of~$\omega$
on monomial quasisymmetric functions.

\begin{definition}
The linear map $\omega$ on quasisymmetric functions is defined by the
following action on monomial quasisymmetric functions.
\begin{equation}
\label{eq:omega}
\omega M_\beta := (-1)^{|\beta|} \sum_{\alpha\le \beta} M_\alpha.
\end{equation}
\end{definition}

It is known (e.g., see the proof of~\cite[Theorem 4.2]{stanley-xg})
that the usual map~$\omega$ is characterized by the
equation $\omega F_\alpha = F_{\overline\alpha}$,
so the following proposition confirms that
our definition of $\omega$ coincides with the standard one.

\begin{proposition}
$\omega F_\alpha  = F_{\overline\alpha}$.
\end{proposition}

\begin{proof}
Applying $\omega$ to Equation~(\ref{eq:fundamental})
and invoking Equation~(\ref{eq:omega}) yields
\[
   \omega F_\alpha = \sum_{\beta\ge\alpha} \omega M_\beta
   = \sum_{\beta\ge\alpha} (-1)^{|\beta|} \sum_{\gamma\le\beta} M_\gamma.
\]
So the coefficient of~$M_\gamma$ in $\omega F_\alpha$ is
\[
   \sum_{\beta: \text{($\beta\ge\alpha$ and $\beta\ge\gamma$})}
     (-1)^{|\beta|}
   = \sum_{\beta \ge \alpha \cup \gamma} (-1)^{|\beta|} = 
\begin{cases}
1, &\text{if $\overline{\alpha \cup \gamma} = \varnothing$;}\\
0, &\text{otherwise.}
\end{cases}
\]
But $\overline{\alpha\cup \gamma} = \varnothing$ is equivalent to
$\gamma\ge\overline\alpha$, so
$\omega F_\alpha = \sum_{\gamma\ge\overline\alpha} M_\gamma =
F_{\overline\alpha}$.
\end{proof}

We are ready for the reciprocity theorem for~$\Xi_D(t)$.

\begin{theorem}
\label{thm:reciprocity}
Let $D$ be a digraph whose vertex set~$V$ is a subset of~$\mathbb{P}$.
Then $\omega\Xi_D(\mathbf{x},t) = \Xi_{\overline D}(\mathbf{x}, t)$.
\end{theorem}

\begin{proof}
Let us apply $\omega$ to both sides of Equation~(\ref{eq:xidt})
and invoke the definition of~$\omega$.
\[
\omega \Xi_D(\mathbf{x},t) = \sum_{(\sq,\beta)} t^{\asc \sq} \omega M_\beta
   = \sum_{(\sq,\beta)} t^{\asc \sq} (-1)^{|\beta|}
      \sum_{\alpha\le\beta} M_\alpha.
\]
Now we interchange the order of summation; i.e., we want to
compute the coefficient of $M_\alpha$ in $\omega \Xi_D(t)$.
This involves a sum over all ordered path covers $(\sq,\beta)$
such that $\beta\ge\alpha$.
The summands involving a fixed sequencing~$\sq$ are
\begin{equation}
\label{eq:alternatingsum}
   \sum_{\beta\ge\alpha} t^{\asc \sq} (-1)^{|\beta|}
    = t^{\asc \sq} \sum_{\beta\ge\alpha} (-1)^{|\beta|}.
\end{equation}
Now note that if $(\sq,\alpha)$ is an ordered path cover
and $\beta$ is any composition such that $\beta\ge\alpha$,
then $(\sq,\beta)$ is also an ordered path cover,
because deleting an edge from a directed path simply
subdivides it into two smaller directed paths.
Therefore the alternating sum in Equation~(\ref{eq:alternatingsum})
is zero unless the only $\beta\ge\alpha$
for which $(\sq,\beta)$ is an ordered path cover
is the maximal composition ($\beta_i = 1$ for all~$i$),
in which case the alternating sum equals one.
But this condition is equivalent to the condition
that there is no directed edge in~$D$ between
any consecutive vertices in the sequencing~$\sq$
that are in the same segment of~$\alpha$,
i.e., that $(\sq,\alpha)$ is an ordered path cover of~$\overline D$.
Finally, note that the definition of $\asc \sq$
is invariant under taking complements of the digraph.
The theorem follows.
\end{proof}

Theorem~\ref{thm:reciprocity} gives us
a nice combinatorial interpretation of $c_{d,\lambda}(\bm)$.

\begin{corollary}
\label{cor:asc}
Let $\bm$ be a Hessenberg function, and
let $D(\bm)$ denote the digraph on~$[n]$
that has an edge $u\to v$ if and only if $v\prec u$ in~$P$.
Then for any composition $\alpha$ whose parts are a permutation of
the parts of~$\lambda$, $c_{d,\lambda}(\bm)$ equals
the number of ordered path covers $(\sq,\alpha)$ of $\overline{D(\bm)}$
with $\asc \sq = d$.
\end{corollary}

\begin{proof}
By Proposition~\ref{prop:xequalsxi},
we know that $X_{G(\bm)}(t) = \Xi_{D(\bm)}(t)$.
By Theorem~\ref{swsym}, 
$X_{G(\bm)}(t)$ is actually a \emph{symmetric} function (whose coefficients
are polynomials in~$t$).
Therefore $\omega X_{G(\bm)}(t) = \omega\Xi_{D(\bm)}(t)$
is also a symmetric function, and the coefficient of $t^d m_\lambda$
equals the coefficient of $t^d M_\alpha$ for any composition~$\alpha$
whose parts are a permutation of the parts of~$\lambda$.
The result then follows from Theorem~\ref{thm:reciprocity}.
\end{proof}

\begin{corollary}
\label{cor:des}
For a sequencing~$\sq$ of a digraph
whose vertex set is a subset of~$\mathbb{P}$,
let the definition of $\des \sq$ be the same as
the definition of $\asc \sq$ except with
``$u<v$'' replaced by ``$v<u$.''
then Corollary~\ref{cor:asc} holds with $\des \sq$ in place of $\asc \sq$.
\end{corollary}

\begin{proof}
For a proper coloring $\kappa$ of a graph
whose vertex set is a subset of~$\mathbb{P}$,
let the definition of $\des\kappa$ be the same as
the definition of $\asc\kappa$ except with
``$\kappa(u)<\kappa(v)$'' replaced by ``$\kappa(u) > \kappa(v)$.''
Shareshian and Wachs prove~\cite[Corollary 2.7]{shareshian-wachs}
that the value of $X_G(t)$ is unchanged if ``asc'' is replaced by ``des.''
It is readily checked that the proofs of
Proposition~\ref{prop:xequalsxi} and Theorem~\ref{thm:reciprocity}
go through if ``asc'' is replaced by ``des'' everywhere.
\end{proof}

As we remarked before, Corollaries \ref{cor:asc} and~\ref{cor:des}
can be derived from Shareshian--Wachs
\cite[Theorem 3.1]{shareshian-wachs}, but 
we have taken our approach because we believe that
Theorem~\ref{thm:reciprocity} is of independent interest.

\section{Betti numbers of regular Hessenberg varieties}\label{sbetti}

The main result of this section is that if
$\sH(\bm,s)$ is a regular Hessenberg variety
and $s$ has Jordan type~$\lambda$,
then its Betti number $\beta_{2d}$ 
equals $c_{d,\lambda}(\bm)$.

Tymoczko~\cite[Theorem 7.1]{tymoczko-linear}
has already done a lot of the work needed to prove
this result, by showing that Hessenberg varieties admit a paving
(or cellular decomposition) by affine spaces,
and obtaining a combinatorial interpretation
of the dimensions of the cells.
For regular Hessenberg varieties,
Tymoczko's theorem simplifies as follows.
If $\lambda$ is an integer partition of~$n$ then
by a \emph{tableau of shape~$\lambda$}
we mean any filling of the boxes of the Young diagram of~$\lambda$
with one copy each of the numbers $1,2,\ldots, n$.

\begin{theorem}[Tymoczko]
\label{thm:tymoczko}
Let $\sH(\bm,s)$ be a regular Hessenberg variety
and let the partition $\lambda$ encode the sizes of the Jordan blocks of~$s$.
Then $\sH(\bm,s)$ is paved by affines.
The nonempty cells are in bijection with tableaux~$T$ of shape~$\lambda$
with the property that $k$ appears in the box immediately to the left of~$j$
only if $k\le m_j$.  The dimension of a nonempty cell is the sum of:
\begin{enumerate}
\item the number of pairs $i,k$ in~$T$ such that
\begin{enumerate}
\item $i$ and $k$ are in the same row,
\item $i$ appears somewhere to the left of~$k$,
\item $k<i$, and
\item if $j$ is in the box immediately to the right of~$k$ then $i\le m_j$;
\end{enumerate}
\item the number of pairs $i,k$ in~$T$ such that
\begin{enumerate}
\item $i$ appears in a lower row than $k$, and
\item $k < i \le m_k$.
\end{enumerate}
\end{enumerate}
\end{theorem}

It remains for us to establish a correspondence between the combinatorics
of Theorem~\ref{thm:tymoczko} and the combinatorics of~$\omega X_{G(\bm)}(t)$,
or equivalently (by the results of the previous section) the
combinatorics of ordered path covers.

\begin{definition}
  If $X$ is a topological space and $i$ is an integer, we write
$\beta_i$ or $\beta_i(X)$ for the $i$th Betti number $\dim \HH^i(X,\mathbb{C})$
of $X$.  
\end{definition}

\begin{theorem}
\label{thm:betti}
Let $\sH(\bm,s)$ be a regular Hessenberg variety and let the Jordan
type of~$s$ be~$\lambda$.  Then the Betti number $\beta_{2d}$ of
$\sH(\bm,s)$ equals $c_{d,\lambda}(\bm)$, and $\beta_i=0$ for $i$ odd.
\end{theorem}

\begin{proof}
As Tymoczko~\cite[Proposition 2.2]{tymoczko-linear} mentions,
it is well known that if we have a paving by affines,
then $\beta_{2d}$ is just the number of nonempty cells
with dimension~$d$.
Furthermore, $\beta_i=0$ for $i$ odd
\cite[Corollary 6.2]{tymoczko-linear}.
On the other hand, by Corollaries \ref{cor:asc} and~\ref{cor:des},
we know that $c_{d,\lambda}(\bm)$ is the number of
ordered path covers $(\sq,\alpha)$ of $\overline{D(\bm)}$
with $\des \sq = d$,
where we may take the parts of the composition~$\alpha$
to be any permutation of the parts of~$\lambda$.
So it suffices to show, firstly,
that there is a bijection between
nonempty cells and ordered path covers $(\sq,\alpha)$,
and secondly, that under this bijection,
the dimension of the nonempty cell is equal to $\des\sq$.

First we should specify~$\alpha$.
If $\lambda$ has $\ell$ parts $\lambda_1, \ldots, \lambda_\ell$,
we set $\alpha_i := \lambda_{\ell+1-i}$.
That is, the parts of~$\alpha$ are the parts of~$\lambda$
in \emph{reverse} order.

Instead of nonempty cells, we use the tableaux~$T$
of Theorem~\ref{thm:tymoczko} to describe our bijection.
Given an ordered path cover $(\sq,\alpha)$ of $\overline{D(\bm)}$,
take the elements of the $i$th path
\[ \sq(\alpha_{i-1}+1)\to \sq(\alpha_{i-1}+2)\to \cdots\to \sq(\alpha_i)\]
and place them from left to right in the $i$th row (from the bottom) of~$T$.
We need to verify that Tymoczko's condition $k\le m_j$
is equivalent to the condition that $k\to j$ is
a directed edge in $\overline{D(\bm)}$.
By definition, there is a directed edge $k\to j$ in $\overline{D(\bm)}$
if and only if there is \emph{not} a directed edge $k\to j$ in~$D(\bm)$,
i.e., if and only if either $k$ and~$j$ are incomparable in~$P(\bm)$,
or $k\prec j$ in~$P(\bm)$.
The only way this property can fail is if $j\prec k$ in~$P(\bm)$,
i.e., if $k > m_j$.  So indeed the conditions are equivalent.

Let us call a pair $i,k$ satisfying the conditions
in Theorem~\ref{thm:tymoczko} a ``T-inversion.''
Using the above bijection, we can think of T-inversions
as certain pairs $i, k$ in an ordered path cover $(\sq,\lambda)$.
The statistic $\des$ can also be thought of as counting certain pairs
$i, k$ of $(\sq,\lambda)$, namely those satisfying
\begin{enumerate}
\item either $i\to k$ and $k\to i$ are both edges of~$\overline{D(\bm)}$
or neither is,
\item $i>k$, and
\item $k$ appears later in the sequencing~$\sq$ than $i$ does.
\end{enumerate}
Call such a pair an ``SW-inversion.''
We claim that for any ordered path cover,
the number of T-inversions equals the number of SW-inversions.
This will prove the theorem.

First let us note that the condition $k<i$ implies
that $k\le m_i$ (since $\bm$ is a Hessenberg function)
and therefore, by the argument we gave above,
$k\to i$ is an edge of $\overline{D(\bm)}$.
That is, if $k<i$ then it is not possible
for neither $i\to k$ nor $k\to i$ to be an edge of $\overline{D(\bm)}$,
so in fact both must be, and in particular we must have $i\to k$,
or in other words $i\le m_k$.
Therefore an SW-inversion can be redefined as a pair $i,k$ such that
\begin{enumerate}
\item $i$ appears earlier in the sequencing~$\sq$ than $k$ does, and
\item $k < i \le m_k$.
\end{enumerate}
It is now immediate that if
$i$ and~$k$ are in different paths then
$i,k$ is a T-inversion if and only if $i,k$ is an SW-inversion,
because by construction, $i$ appearing in an earlier path than~$k$
is equivalent to appearing in a lower row than~$k$ in the tableau.

If $i$ and $k$ are in the same path then the situation
is more complicated because T-inversions and SW-inversions
do not necessarily coincide.
However, we now give a bijection from the set of SW-inversions
to the set of T-inversions, thereby showing that they are equinumerous.

Given an SW-inversion $i,k$, let
$k_1, k_2, \ldots, k_r$ denote the remaining elements, in order,
that succeed $k$ in the path.
For convenience, set $k_0 := k$ and $k_{r+1} := \infty$.
Now let $j$ be the smallest number such that $i\le m_{k_{j+1}}$.
Then we claim that $i, k_j$ is a T-inversion, and that this is a bijection.

First let us verify that $i, k_j$ is a T-inversion.
Condition 1(d) is satisfied almost by definition
because what the construction is doing is scanning to the right
\emph{until} condition 1(d) is satisfied,
and it will always succeed, since we just take $j=r$ in the worst case.
So we just need to verify that $i>k_j$.  If $j=0$ then we are done,
because $(i,k_0) = (i,k)$ is an SW-inversion by assumption,
and in particular $i>k$.  Otherwise, by minimality of~$j$,
we know that $i>m_{k_j} \ge k_j$.

Thus the construction scans rightwards from~$k$
\emph{until the first T-inversion $i, k_j$ is reached.}

To see that this map is injective, observe that by minimality of~$j$,
we have $i>m_{k_{j'}}$ for every $0\le j' \le j$, so
$(i,k_{j'})$ is \emph{not} an SW-inversion.
Thus, as we scan rightwards from~$k$ in search of the first T-inversion
$i, k_j$, we do not encounter any other SW-inversions en route.
If more than one SW-inversion were mapped to the same T-inversion,
then the leftmost one would have to cross over the other ones en route.

To see that the map is surjective, we can define an inverse map,
that scans \emph{leftwards} from a T-inversion until it finds a pair
that satisfies $i\le m_k$.  Such a scan always succeeds because
in the worst case it ends up at the successor $i'$ of~$i$, and
$i\le m_{i'}$ because they are consecutive elements of a path.
Then by minimality, if we arrive at a pair $i,k$ with $k'$
being the successor of~$k$, we must have $k>m_{k'}\ge k$,
so what we have arrived at is indeed an SW-inversion.
\end{proof}

Let us remark that our proof shows that at least in the
case of regular Hessenberg varieties,
the two cases of Theorem~\ref{thm:tymoczko} can be unified,
namely that the dimension is just the number of pairs $i,k$
such $i$ appears to the left of~$k$ or in a lower row than~$k$,
and $k<i\le m_k$.

\begin{corollary}\label{pbetti}
  Let $\sH(\mathbf{m},s)$ be a regular Hessenberg variety with $s$ of
  type $\lambda$ as in Theorem~\ref{thm:betti}.   Set 
$$
q=q_{\sH(\mathbf{m},s)}:=\sum_{i\in\mathbb{Z}} \beta_{i} t^{i-|\mathbf{m}|}.
$$
Then $q(t)=q(t^{-1})$.  
\end{corollary}

\begin{proof}
First note that $|\bm|$ (as defined in Equation~(\ref{eq:mdim}))
is the number $|E|$ of edges in the incomparability
graph $G=G(\bm)$ of $P(\bm)$.  This follows directly from the 
description of $G(\bm)$ given in \S\ref{gn}.  
By \cite[Corollary 4.6]{shareshian-wachs},
$X_G(\mathbf{x},t)$ is palindromic.  More precisely, we have
$X_G(\mathbf{x},t)=t^{|\bm|} X_G(x,t^{-1})$.   Therefore, for each partition
$\lambda$, we have 
$$
\sum_d c_{d,\lambda}(\bm) t^{d} = t^{|\bm|} \sum_{d} c_{d,\lambda} t^{-d}.
$$
So 
\begin{align*}
  q(t)&=\sum_i \beta_i t^{i-|\bm|}= \sum_{d} c_{d,\lambda}(\bm) t^{2d-|\bm|}\\
&=t^{-|\bm|}\sum_{d} c_{d,\lambda}(\bm)  t^{2d}=
t^{-|\bm|}t^{2|\bm|}\sum_{d} c_{d,\lambda}(\bm) t^{-2d}\\
&=t^{|\bm|}\sum_{d} c_{d,\lambda}(\bm) t^{-2d}
=\sum_d c_{d,\lambda}(\bm) t^{|\bm|-2d}\\
&=q(t^{-1}).
\end{align*}
\end{proof}

\begin{corollary}\label{dimcor}
  Suppose $s$ is a regular matrix.  Then, for all $i\in\mathbb{Z}$,
  we have $\beta_i=\beta_{2|\mathbf{m}|-i}$.  Consequently, 
  $\dim\sH(\mathbf{m},s)=|\mathbf{m}|$.
\end{corollary}
\begin{proof}
  The first assertion follows (after a few algebraic manipulations) from
  Corollary~\ref{pbetti}.  It is well known that, for a complex, projective
  variety $X$, we have $\dim X=\max\{i:\HH^{2i} (X,\mathbb{C}) \neq 0\}$.  So, 
  the second assertion is a direct consequence of the
  first.
\end{proof}

\section{Local monodromy and local fundamental groups}\label{lmlfg}
\subsection{Local systems} In this subsection, we review some
terminology concerning local systems.  This material is standard
(going back in some ways to Riemann~\cite{riemann-abel}), but we
realized that including it might help to make our paper more broadly
accessible.  Moreover, since we are making considerable use of local
systems, it seems appropriate to be as precise as possible about what
they are.  To have a specific (modern) reference, we follow the
dictionary on page 3 of Deligne's book on differential
equations~\cite{deligne-eqdiff}.

By a locally constant sheaf on a topological space $X$, we simply mean
a sheaf of sets $\mathcal{F}$ which is locally isomorphic to a
constant sheaf of sets.  In other words, each $x\in X$ has an open
neighborhood $U$ such that the restriction of $\mathcal{F}$ to $U$ is
constant.  We can consider the class of locally constant sheaves as a
full subcategory of the class of all sheaves.  On the other hand, it
is well known and easy to see that the category of locally constant
sheaves on $X$ is equivalent to the category of covering spaces of $X$
(c.f.~\cite[Definition and Proposition 3.41]{wedhorn}).

By a local system on a topological space $X$, we mean a sheaf of finite
dimensional 
$\mathbb{C}$-vector spaces $\mathcal{F}$ on $X$ which is locally 
isomorphic to a constant sheaf of $\mathbb{C}$-vector spaces. 
Note that this definition differs slightly from Deligne's
in~\cite{deligne-eqdiff} in that Deligne requires the dimension of the stalks to be constant.
However, this is guaranteed by our definition if $X$ is connected, which
is the most important case, and
the added flexibility is useful.

For any ring $R$, we could equally well define $R$-local systems
by replacing 
$\mathbb{C}$ with $R$ (and finite dimensional vector spaces by finitely 
generated $R$-modules).  
But, to avoid cluttering up the
notation, we refer the reader to~\cite{ElZeinLS} for this notion.
(For most of the paper we only use $\mathbb{C}$-local systems, however,
 we do use $\mathbb{Z}$-local systems in \S\ref{sghs} and $A$-local systems
for $A$ a polynomial ring in \S\ref{mtd}.)

We view the class of local systems as a
full-subcategory of the category of sheaves of $\mathbb{C}$-vector
spaces on $X$.  Clearly, there is a forgetful functor from the
category of local systems on $X$ to the category of locally constant
sheaves on $X$ (by forgetting the $\mathbb{C}$-vector space
structures).

Suppose now that $X$ is non-empty.  Pick a point $x_0\in X$, which we call a
``basepoint.''  Then the fundamental group $\pi_1(X,x_0)$ acts on the
fiber $\mathcal{F}_{x_0}$ of any locally constant sheaf $\mathcal{F}$
giving us a homomorphism
$\rho:\pi_1(X,x_0)\to \Aut \mathcal{F}_{x_0}$.  If $\mathcal{F}$ is a
local system then $\rho$ respects the $\mathbb{C}$-vector space
structure giving us a group homomorphism
\begin{equation}
  \label{e.monodromy}
  \rho:\pi_1(X,x_0)\to \GL(\mathcal{F}_{x_0})
\end{equation}
which is usually called the \emph{monodromy representation}.
The fundamental fact about local systems and locally constant sheaves is then the following standard result (which
can be found on pages 3 and 4 of~\cite{deligne-eqdiff}).

\begin{theorem}\label{t.equiv}
  Suppose that $X$ is a locally path connected, locally simply
  connected, connected topological space equipped with a point $x_0$.
  Then the functor $\mathcal{F}\leadsto \mathcal{F}_{x_0}$ induces an
  equivalence from the category of locally constant sheaves
  (resp.\ local systems) on $X$ to the category of $\pi_1(X,x_0)$-sets
  (resp.\ finite dimensional complex representations of
  $\pi_1(X,x_0)$).
\end{theorem}

\begin{proof}[Sketch]
  Since Deligne does not actually prove Theorem~\ref{t.equiv} in~\cite{deligne-eqdiff}, we give a sketch.
  
  The main point is that, under the assumption that $X$ is locally
  path connected, locally simply connected and connected, there exists
  a universal cover $\tilde X$ of $X$.  For a proof, see the
  discussion starting on page 64 of Hatcher's book~\cite{Hatcher}, where $\tilde{X}$
  is constructed as a space of homotopy classes of paths starting from the point $x_0$.
  Moreover $\pi_1(X,x_0)$ acts freely on $\tilde X$ with quotient $X$.  Given a
  $\pi_1(X,x_0)$-set $E$ (resp.\ a finite dimensional $\pi_1(X,x_0)$-representation $E$),
  we  consider the quotient $\mathcal{F}_E:=(\tilde{X}\times E)/\pi_1(X,x_0)$ where the fundamental
  group acts on the product by $\gamma(\tilde{x},e)=(\gamma\tilde{x},\gamma^{-1}e)$.
  That is, we form the Borel construction, where here $E$ is given the discrete topology.
  
  The space $\mathcal{F}_E$ is naturally a covering space of $X$ via
  the map $\mathcal{F}_E\to X$ induced by projection on the first factor in the product $\tilde{X}\times E$.   If $E$
  is a $\mathbb{C}$-vector space, the sheaf corresponding to $\mathcal{F}_E$ has the natural structure of a sheaf of $\mathbb{C}$-vector spaces. 
  We leave the rest of the verification to the reader.
\end{proof}

\begin{corollary}\label{c.equiv}
  Suppose $\mathcal{F}$ is a local system on a topological space $X$ as in
  Theorem~\ref{t.equiv}.  Then there is a natural isomorphism
  $\HH^0(X,\mathcal{F})=\mathcal{F}_{x_0}^{\pi_1(X,x_0)}$.  
\end{corollary}
\begin{proof}
  Write $\mathbb{C}_X$ for the constant local system on $X$, which corresponds to the
  trivial representation of $\pi_1(X,x_0)$.  Then we have $\mathcal{F}^{\pi_1(X,x_0)}=\Hom(\mathbb{C}_X,\mathcal{F})$ by Theorem~\ref{t.equiv}.  But it is easily seen that the
  natural map $\Hom(\mathbb{C}_X,\mathcal{F})\to \mathcal{F}(X)$ is an isomorphism.
\end{proof}

\begin{remark}
  By Proposition A.4 on page 531 of Hatcher's book~\cite{Hatcher},
  CW complexes are locally contractible.  It follows that, if $X$ is a CW
  complex and $Y$ is a closed subcomplex, then $X\setminus Y$ is locally contractible.  In particular, if $X\setminus Y$ is connected then it satisfies the hypotheses of Theorem~\ref{t.equiv}.  
\end{remark}

\subsection{Local homotopy type}
In this subsection we review the definition and some of the main
properties of local homotopy type.  This material is probably
well known to some readers, but we feel that it will be convenient to
review it.  Our treatment follows the book by
Looijenga~\cite{Looijenga}, a paper by Kumar~\cite{Kumar93} and
another paper by Prill~\cite{Prill}.

Suppose $X$ is a topological space and $x\in X$.  A \emph{fundamental
  system of neighborhoods} $\mathscr{U}$ of $x$ is a system of open
neighborhoods such that any open neighborhood $V$ of $x$ contains a
$U\in\mathscr{U}$.  

The following Lemma is~\cite[Lemma 1.1]{Kumar93}.

\begin{lemma}\label{lm1}
  Suppose $X$ is a CW complex, $x\in X$ and $Y$ is a closed subcomplex
  of $X$ containing $x$.  Then there exists a fundamental system
  $\{U\}_{U\in\mathscr{U}}$ of open neighborhoods of $x$ in $X$ such
  that the following condition is satisfied:
\begin{equation}
\parbox{25em}{For any $U,V\in\mathscr{U}$ with $V\subset U$, the inclusion $V\setminus Y\hookrightarrow U\setminus Y$ is a homotopy  equivalence.}
\end{equation}
\end{lemma}

A system of neighborhoods $\mathscr{U}$ as in Lemma~\ref{lm1} is
called a \emph{good fundamental system of neighborhoods relative to
  $Y$}.

\begin{lemma}\label{lm1.5}
  Suppose $\mathcal{C}$ is a category and 
$$
A\stackrel{f}{\to} B\stackrel{g}{\to} C\stackrel{h}{\to} D
$$
is a sequence of morphisms.  Assume that $g\circ f$ and $h\circ g$ 
are isomorphisms.  Then $f, g$ and $h$ are all isomorphisms.
\end{lemma}
\begin{proof}
Easy exercise.
\end{proof}

We have adapted the proof of the following Proposition from
Looijenga's~\cite[p.~114]{Looijenga}, and Prill's Proposition
2~\cite{Prill}.

\begin{proposition}\label{pgood}
  Suppose $\{\mathscr{U}_{\alpha}\}_{\alpha\in I}$ is a non-empty collection
of good fundamental systems of neighborhoods as in Lemma~\ref{lm1}.  
Then so is $\mathscr{U}:=\cup_{\alpha\in I} \mathscr{U}_{\alpha}$. 
Consequently, the union of all good fundamental systems of neighborhoods
is itself a good fundamental system of neighborhoods. 
\end{proposition}

\begin{proof}
Take $U\in\mathscr{U}_{\alpha}$ and $V\in\mathscr{U}_{\beta}$ with $V\subset U$.
We can find $U'\in\mathscr{U}_{\alpha}$ such that $U'\subset V$, 
and $V'\in
    \mathscr{U}_{\beta}$ such that $V'\subset U'$.  Then apply Lemma~\ref{lm1.5}
to the sequence of inclusions
$$
V'\setminus Y\to
U'\setminus Y 
\to V\setminus Y
\to U\setminus Y.
$$
This shows that the inclusion $V\setminus Y\to U\setminus Y$ is a homotopy 
equivalence.
\end{proof}

\begin{definition}\label{lhtdef}
  Suppose $X$ is a CW complex and $Y$ is a closed subcomplex
  containing a point $x$.  We say an open neighborhood $U$ of $x$ is
  \emph{good relative to $Y$} if $U$ is an element of a good
  fundamental system of neighborhoods.   The \emph{local homotopy type
    of $X\setminus Y$ at $x$} is the homotopy type of $U\setminus Y$
  where $U$ is any good neighborhood.
\end{definition}

If $A$ and $B$ are objects in any category $\mathcal{C}$, we say that
$A$ is a \emph{retract} of $B$ if there are morphism $i:A\to B$ and
$r:B\to A$ such that $r\circ i=\id_A$.  In other words, we follow
Mac Lane's terminology in~\cite[p.~19]{MacLane}.

Suppose $U$ is a good neighborhood of $x$ and $W$ is an arbitrary (not
necessarily good) open neighborhood of $x$ contained in $U$.  Then we
can find a good neighborhood $V$ such that $V\subset W$.
Since $V$ is good, Proposition~\ref{pgood} shows that the composition
$$
V\setminus Y\to W\setminus Y\to U\setminus Y
$$
is a homotopy equivalence.  In other words, the local homotopy type of
$X\setminus Y$ at $x$ is a retract of the homotopy type of $W\setminus
Y$.

Now suppose $X$ is an analytic space, $Y$ is a Zariski closed subspace
and $x\in Y$.  We can find an analytic open neighborhood $W$ of $x$ in $X$ such
that $W$ has the topological structure of a CW complex with $W\cap Y$
a subcomplex.  (See, for example, ~\cite{Lojasiewicz64}.)
Consequently, there exist good neighborhoods of $x$ in $X$ relative to
$Y$.

The following fact is certainly well known (see, e.g.,~\cite[Corollary
1]{Prill}), but we give a proof because it is short.

\begin{fact}\label{lmfact}
  Suppose $X$ is a complex manifold and $Y$ is a closed, nowhere dense,
  analytic subspace of $X$ containing a point $x$.  Then $U\setminus
  Y$ is  non-empty and connected for any good neighborhood of $x$.
\end{fact}

\begin{proof}
  Let $U$ be a good neighborhood of $x$ and let $V$ be any connected
  neighborhood of $x$ contained in $U$.  Then $V\setminus Y$ is
  connected (for example, by the Criterion for Connectedness on page
  133 of~\cite{GR}).  It is also non-empty.  But the homotopy type of
  $U\setminus Y$ is a retract of the homotopy type of $V\setminus Y$.  So $U\setminus Y$ is connected
  and non-empty as well.
\end{proof}

If $X$ is smooth at $x$,  we can find a contractible good
neighborhood $U$ of~$x$.  (See~\cite{Prill}.)  In fact, we can 
take a sufficiently small ball as in the following theorem, which 
follows from Theorem 5.1 of Dimca's~\cite{dimca}.

\begin{theorem}
 \label{t-dimca}
  Suppose $X$
  is a complex manifold of dimension $n$
  at $x$
  and $Y$
  is a closed analytic subspace of $X$
  containing $x$.
  For each positive real number $r$,
  write $B_r$
  for the ball of radius $r$
  centered at $0$
  in $\mathbb{C}^n$.
  Then there exists a good neighborhood $U$
  of $x$
  relative to $Y$
  and biholomorphism $\varphi:U\to B_1$
  such that the following holds: For each $r\in (0,1)$,
  $\varphi^{-1}B_{r}$ is a good neighborhood of $x$.
\end{theorem}

\subsection{Local fundamental group}

Fact~\ref{lmfact} leads to the following definition.

\begin{definition}\label{dlfg}
  Suppose $X$ is a complex manifold, and $Y$ is a closed,
  nowhere dense, analytic subspace of $X$ containing a point $x$.
  Then the \emph{local fundamental group} of $X\setminus Y$ at $x$ is the
  isomorphism class of the group $\pi_1(U\setminus Y,p)$ where $U$ is any good
  neighborhood of $x$ (with respect to $Y$) and $p\in U\setminus Y$.
\end{definition}

Since the smoothness of $X$
in the Definition~\ref{dlfg} (together with Fact~\ref{lmfact}) implies
that $U\setminus Y$
is connected, the isomorphism class of $\pi_1(U\setminus Y,p)$
is indeed well-defined.  But, since we have not given a way to fix a
base point, it is only defined up to a non-canonical isomorphism.

On the other hand, suppose $X$ as in Definition~\ref{dlfg} is
connected.  Pick any point $q\in X\setminus Y$. Given a good
neighborhood $U$ of $x$ relative to $Y$, we can find a point $p\in
U\setminus Y$ and a path $\gamma$ from $p$ to $q$.  From this, we get
a group homomorphism
$$
\varphi_{\gamma}:\pi_1(U\setminus Y,p)\to \pi_1(X\setminus Y,q).
$$
Changing $\gamma$ has the effect of conjugating $\varphi_{\gamma}$ by 
an element of $\pi_1(X\setminus Y,q)$.   So the conjugacy class of the 
subgroup $\varphi_{\gamma}(\pi_1(U\setminus Y,p))$ 
is independent of $\gamma$.

\begin{proposition}\label{psec}
  Suppose $f:X'\to X$ is a morphism of complex analytic spaces
  admitting a section $\epsilon:X\to X'$.  Let $Y$ be a closed,
  nowhere dense, analytic subspace of $X$ containing $x\in X$, and set
  $Y':=f^{-1}Y$.  Then the local homotopy type of $X\setminus Y$ at
  $x$ is a retract of the local homotopy type of $X'\setminus Y'$ at
  $\epsilon(x)$.  In particular, if $X$ and $X'$ are complex
  manifolds, then the local fundamental group of $X\setminus Y$ at $x$
  is a retract of the local fundamental group of $X'\setminus Y'$ at
  $\epsilon(x)$.
\end{proposition}
\begin{proof}
  Pick a good neighborhood $U$ of $x$.  Then find a good neighborhood
  $V$ of $\epsilon(x)$ contained in $f^{-1}U$.  Finally, find a good
  neighborhood $U'$ of $x$ contained in $\epsilon^{-1}V$.  We then 
  have a composition
$$
U'\setminus Y\stackrel{\epsilon}{\to} V\setminus Y'\stackrel{f}{\to} U\setminus Y
$$
which is a homotopy equivalence.  The result follows.  
\end{proof}

\subsection{Local systems and local invariant cycles}  
Suppose $X$ is a CW complex containing a closed subcomplex $Y$ which
contains a point $x$, and $\mathcal{L}$ is a local system of complex
vector spaces on $X\setminus Y$.  For any two good neighborhoods $U_1$
and $U_2$ of $x$ and any integer $k$, the sheaf cohomology groups
$\HH^k(U_i\setminus Y,\mathcal{L}), i=1,2$ are canonically isomorphic.
To see this, take a good neighborhood $V\subset U_1\cap U_2$, and note
that the restriction maps $\HH^k(U_i\setminus Y,\mathcal{L})\to
\HH^k(V\setminus Y,\mathcal{L})$ are isomorphisms.  So we write 
$\HH^k(x,\mathcal{L})$ for the group $\HH^k(U\setminus Y,\mathcal{L})$
where $U$ is any good neighborhood of $x$.   It is isomorphic to the
group $\colim \HH^k(U\setminus Y,\mathcal{L})$ where the colimit is taken 
over all open neighborhoods of $x$. 
The group $\HH^0(x,\mathcal{L})$ is
called the space of \emph{local invariants}.

If $X$ is a complex manifold and $Y$ is a nowhere dense analytic subspace, then
$U\setminus Y$ is connected for any good neighborhood of $x$ relative
to $Y$.  Pick a basepoint $p\in U\setminus Y$.  Then the data of the
local system $\mathcal{L}$ defines an action of $\pi_1(U\setminus
Y,p)$ on the fiber $\mathcal{L}_p$ at $p$.  Moreover, via Corollary~\ref{c.equiv}, the space of
local invariants is given by the invariants of the action:
\begin{equation}
  \label{fgi}
  \HH^0(x,\mathcal{L})=\mathcal{L}_p^{\pi_1(U\setminus Y,p)}.
\end{equation}

\begin{corollary}
\label{cor-nowhere dense}
  Suppose $X$ is smooth and $B$ is any connected neighborhood of $x$
contained in a good neighborhood $U$.  Then $\HH^0(x,\mathcal{L})=
\HH^0(B\setminus Y,\mathcal{L})$.
\end{corollary}

\begin{proof}
Pick a point $b\in B\setminus Y$. 
We have $\HH^0(B\setminus Y,\mathcal{L})=\mathcal{L}_b^{\pi_1(B\setminus Y,b)}$.
But $\pi_1(U\setminus Y,b)$ is a retract of $\pi_1(B\setminus Y, b)$.  
So $\mathcal{L}_b^{\pi_1(B\setminus Y,b)}=\mathcal{L}_b^{\pi_1(U\setminus Y,b)}=
\HH^0(x,\mathcal{L})$.  
\end{proof}

We can also describe the space $\HH^k(x,\mathcal{L})$ sheaf
theoretically.  Write $j:X\setminus Y\to X$ for the inclusion.  Then
the group $\HH^k(x,\mathcal{L})$ is naturally isomorphic to
$(R^kj_*\mathcal{L})_x$; i.e., to the stalk at $x$ of the $k$th
higher direct image $R^kj_*\mathcal{L}$.

The following is certainly well known, but we sketch a short proof.

\begin{lemma}\label{lsurj}
  Suppose $X$ is a connected complex manifold and $Y$ is
  a nowhere dense closed analytic subspace.  Then, for $p\in X\setminus Y$, the
  homomorphism $\pi_1(X\setminus Y,p)\to \pi_1(X,p)$ is surjective.
\end{lemma}

\begin{proof}[Sketch]
Let $\pi:\tilde{X}\to X$ denote the universal cover of $X$.  Then 
$\pi^{-1}(X\setminus Y)=\tilde{X}\setminus \pi^{-1}(Y)$ is connected because $\tilde{X}$ 
is a complex manifold and $\pi^{-1}(Y)$ is a closed, nowhere dense, 
complex analytic subspace~\cite[p.~133]{GR}.
It follows that $\pi_1(X\setminus Y,p)$ acts transitively on $\pi^{-1}(p)$.
If we pick a point $\tilde p$ in $\pi^{-1}(p)$, we get an identification
of $\pi^{-1}(p)$ with $\pi_1(X,p)$.   Moreover, the action of $\pi_1(X,p)$
on $\pi^{-1}(p)$ corresponds to the left regular action of $\pi_1(X,p)$ on
itself.   From this, we see that  
the action of $\pi_1(X\setminus Y,p)$ on $\pi_1(X,p)$
induced by the group homomorphism $\pi_1(X\setminus Y,p)\to \pi_1(X,p)$
is transitive.  
Therefore the 
map of fundamental groups is surjective.
\end{proof}

Now, suppose $X$ is a connected, complex manifold, $Y$ is a closed,
nowhere dense, analytic subspace, $\mathcal{L}$ is a local system on
$X\setminus Y$, $x\in Y$ and $q\in X\setminus Y$.  The \emph{monodromy
  group of $\mathcal{L}$} is the image $M$ of the group homomorphism
$\pi_1(X\setminus Y,q)\to \mathbf{GL}(\mathcal{L}_q)$.  By
Lemma~\ref{lsurj}, $M$ is unchanged if we replace $Y$ by a larger
closed, nowhere dense analytic subset $Y'$.  That is, if $Y'$
contains $Y$ (but not $q$), then the image of the homomorphism
$\pi_1(X\setminus Y',q)$ is also $M$.

Suppose $U$ is a good fundamental neighborhood of $x$ relative
to $Y$ and $p\in U\setminus Y$. 
Then the \emph{local monodromy group of
  $\mathcal{L}$ at $y$} is the image $H=H(y)$ of the composition
$$
\pi_1(U\setminus Y,p)\stackrel{\varphi_{\gamma}}{\to} \pi_1(X\setminus Y,q)\to M
$$
where $\gamma$ is a path from $p$ to $q$.  It depends on the choice of $U$, $\gamma$ and $p$, but only up to conjugacy by an element of $M$. 
Like $M$ itself, $H$ is independent of $Y$ in the sense that enlarging $Y$
does not change $H$. 

\section{Palindromic Betti numbers and the local invariant cycle
theorem}\label{sbn}

\subsection{Main Theorems}
\label{ss-main}
A crucial tool in our argument is the
local invariant cycle theorem of
Beilinson, Bernstein, and Deligne (BBD), which we state here in the
generality relevant to this paper.  

\begin{theorem}[{\cite[Corollaire 6.2.9]{bbd}}]
\label{thm:bbd}
Suppose $f:X\arr Y$ is a proper morphism of smooth, separated,
irreducible complex schemes.  Let $y\in Y(\mathbb{C})$,
and set $X_y:=f^{-1}(y)$.
Suppose that $U$ is a Zariski dense, Zariski open subset of $Y$
such that the restriction of $f$ to $f^{-1}U$ is smooth.  Then, for
every sufficiently small ball $B=B(y)$  centered at  $y$ as in Theorem~\ref{t-dimca},
the natural map 
\begin{equation}\label{e.lic}
  \HH^i(X_y,\mathbb{C})\arr
\HH^0(B(y)\cap U, R^if_*\mathbb{C})
\end{equation}
is a surjection.
Moreover,  $B(y)\cap U$ is nonempty, and  we have
\begin{equation}
  \label{e.lic-pi1}
  \rH^0(B(y)\cap U, R^if_*\mathbb{C})=\rH^i(X_z)^{\pi_1(B(y)\cap U,z)}
\end{equation}
for any $z\in B(y)\cap U$.   
\end{theorem}

The vector space $\HH^0(B(y)\cap U, R^if_*\mathbb{C})$ is called the
space of \emph{local invariant cycles}, and we call the map 
in~\eqref{e.lic} (which we will explain in some detail below) the \emph{local invariant
  cycle map}.  The assumption that $f$ is smooth and proper over $U$
implies that the sheaves $R^if_*\mathbb{C}$ restrict to local systems
on $U$. 
It follows from Corollary~\ref{cor-nowhere dense} that, for a fixed $U$, the spaces
$\HH^0(B(y)\cap U, R^if_*\mathbb{C})$ are canonically isomorphic for
$B(y)$ sufficiently small.
Moreover, as BBD point out, up to a
canonical isomorphism,~\eqref{e.lic} is independent of $U$. 

When $Z$ is a scheme, we write $d_Z:=\dim Z$ to save space.  This notation
is useful in the main result of this section, which is the following.

\begin{theorem}\label{t1}
  Suppose that $f:X\arr Y$ is a projective morphism 
 between smooth, separated, irreducible complex schemes; and let $y$ be a closed point of~$Y$.  
Set $d=d_X-d_Y$. 
Then the local
  invariant cycle map~\eqref{e.lic} is an isomorphism for all $i\in\mathbb{Z}$ if and only if
  \begin{equation}
\label{e.t1}
\dim \HH^i(X_y,\mathbb{C})=\dim \HH^{2d-i}(X_y,\mathbb{C})
\end{equation}
for all $i$.  
\end{theorem}

The rest of this section is devoted to a proof of Theorem~\ref{t1}.
Our proof uses the ideas behind the proof of Theorem~\ref{thm:bbd}
extensively.  Somewhat unfortunately for us, in~\cite{bbd}, the proof
of Theorem~\ref{thm:bbd} in the case of complex varieties is more or
less left to the reader to construct using the proof given earlier in
the book of the $\ell$-adic analogue of the theorem for schemes of
finite type over a finite field.  While it is probably fairly clear 
how to do this for 
anyone who has made it to the last few pages of~\cite{bbd} 
where Theorem~\ref{thm:bbd} appears, it makes it difficult for us to cite 
passages in the text where specific results we need are proved. 
In the original arXiv version of this paper~\cite{BrosnanChow},
we handled this essentially by assuming that the reader was familiar
with the proof of Theorem~\ref{thm:bbd}.   However, we realized that this 
approach has serious disadvantages. 
So, to help make our proof of
Theorem~\ref{t1} as clear and precise as possible, we have decided to
include a proof of Theorem~\ref{thm:bbd} in the complex case.

One advantage of this is that we are able to use the Kashiwara
conjecture for semisimple perverse sheaves~\cite{KashiwaraConj}, which has been proved by T.~Mochizuki and independently by
combining work of A.J. de Jong, Drinfeld, Gaitsgory, and
B\"ockle--Khare~\cite{mochizuki-1, mochizuki-2, dJ-conj,
  drinfeld-kashiwara, gaitsgory-dj, boekle-khare-2}.  This allows us
to point out strong forms of Theorems~\ref{thm:bbd} and
Theorem~\ref{t1}.  See Theorem~\ref{big-lic} and
Theorem~\ref{t-pal-main} below.

\subsection{The local invariant cycle map}
\label{sub-lic}
\subsubsection{Geometric definition}  
\label{subsub-geom}
The map~\eqref{e.lic} in the statement of Theorem~\ref{thm:bbd} can be defined in two equivalent ways, geometrically and
sheaf theoretically.  We start with the geometric component of the
definition.  To explain it, write $X_S$ for $f^{-1}S$ when
$S\subset X$, and write $f_S$ for the map $X_S\to S$ coming from the
restriction of $f$.  Then, for $B=B(y)$ a sufficiently small ball, the
restriction morphism
\begin{equation}
  \label{e.resmor}
  \rH^i(X_B,\mathbb{C})\to \rH^i(X_y,\mathbb{C})
\end{equation}
is an isomorphism.   This follows from proper base change. 
On the other hand, we have a map 
\begin{equation}
  \label{e.ss}
  \rH^i(X_{B\cap U},\mathbb{C})\to \rH^0(B\cap U, R^if_*\mathbb{C})
\end{equation}
coming from the edge homomorphism in the Leray--Serre spectral sequence applied
to the fibration  $f_{B\cap U}:X_{B\cap U}\arr B\cap U$.  
Composing the map in \eqref{e.ss} with the inverse of the map in~\eqref{e.resmor} gives the local invariant cycle map~\eqref{e.lic}. 

\subsubsection{General definition}  
\label{gen-def}

Theorem~\ref{thm:bbd} is proved (and even stated) in~\cite{bbd} in a
much more general sheaf-theoretic context.  This allows for greater
generality in the statements, but it also allows for more flexibility
in the proof.  As we will use this generality to prove
Theorem~\ref{t1}, we now explain how to generalize the local invariant cycle map
for complexes of sheaves on $Y$.  This essentially involves unwinding the
definition of the edge homomorphism in the hypercohomology spectral sequence.

Suppose $Y$ is any scheme of finite type over $\mathbb{C}$.  We write
$\Dbc Y$ for the bounded derived category of sheaves of complex vector
spaces on $Y$ with constructible cohomology.  Given a complex
$K\in\Dbc Y$ and an integer $i$, we write $H^iK$ for the $i$th
cohomology sheaf of $K$.  For $j\in\mathbb{Z}$, we write $K[j]$ for
the shift of~$K$ by $j$ units to the left.  So $H^i K[j]=H^{i+j}K$.
If $y\in Y(\mathbb{C})$ and $\mathcal{F}$ is a sheaf on~$Y$, then, as
usual, $\mathcal{F}_y$ denotes the stalk of $\mathcal{F}$ over $y$.
Similarly, if $K$ is a complex, $K_y$ denotes the object in the derived
category of $\mathbb{C}$ vectors spaces obtained by taking stalks.
Since taking stalks is exact, we have a canonical isomorphism
$H^i(K_y)=(H^iK)_y$.  Usually, we simply write $H^iK_y$ for this vector space.

Now, let $j:U\hookrightarrow Y$ denote the inclusion of a Zariski open
subset and let $y\in Y(\mathbb{C})$ be a closed point.  Adjunction then
gives us maps 
\begin{align}
  \lambda:H^i K&\arr j_*j^* H^i K\label{e.lambdagen}\\
  \lambda(y):H^i K_y&\arr (j_*j^* H^i K)_y\label{e.lambday}
\end{align}
which we call the \emph{(generalized) local invariant cycle maps}. 
Here we get \eqref{e.lambday} from \eqref{e.lambdagen} by taking stalks,
and $j_*$ denotes the pushforward of sheaves (not the derived pushforward
as it often does in~\cite{bbd}).
The functor $j^*$ is just the restriction to $U$.  So for $M\in\Dbc Y$, 
$j^*M$ is synonymous with $M_{|U}$. 

When $K=Rf_*\mathbb{C}$ as in Theorem~\ref{thm:bbd}, then 
$H^iK_y=\HH^i(X_y,\mathbb{C})$ by proper base change, and 
$(j_*j^* H^i K)_y =(j_*j^* R^i f_*\mathbb{C})_y$ which is equal to 
$\rH^0(B(y)\cap U,R^if_*\mathbb{C})$ for $B(y)$ sufficiently small 
by the constructibility of the sheaves involved.  (Compare with 
the statement of the local invariant cycle theorem in~\cite[Corollaire 6.2.9]{bbd}).  Moreover, it is easy to see that $\lambda(y)$ agrees with
the geometric description of~\eqref{e.lic} in \S\ref{subsub-geom}.

Note that $\lambda$ and $\lambda(y)$ are natural in $K$.  To make this
explicit, write $\Shv_c Y$ for the category of constructible sheaves
of $\mathbb{C}$ vector spaces on $Y$.  Then $K\leadsto H^iK$ and
$K\leadsto j_*j^*H^iK$ are both additive functors from $\Dbc Y$ to
$\Shv_c Y$, and $\lambda$ is a natural transformation from the first
to the second (as it comes from the adjunction, which is itself natural).
So write $\lambda_K^i$ for the map in  \eqref{e.lambdagen} to keep track of the
index and the complex.  Then,  $\lambda^i_{K[j]}= \lambda^{i+j}_{K}$, and,  for $K_1,K_2\in\Dbc Y$, $\lambda^i_{K_1\oplus K_2}=\lambda^i_{K_1}\oplus
\lambda^i_{K_2}$.   Similar remarks hold obviously for $\lambda(y)$. 

Following~\cite{bbd}, we are going to isolate a class of objects $K$
in $\Dbc Y$ on which~$\lambda$, and thus $\lambda(y)$, turn out to be
surjections.  However, it might help to start out with a trivial example
along with a trivial non-example.

\begin{example}
  Let $Y=\mathbb{A}^1_{\mathbb{C}}$, the affine line and let
  $j:U\hookrightarrow Y$ denote the inclusion of the complement of the
  origin.  Consider the sheaves $\mathcal{F}=j_! \mathbb{C}_U$ and
  $\mathcal{G}=\mathbb{C}_Y$ as objects in $\Dbc Y$.  
For $i\neq 0$, both the source and target of $\lambda$ are $0$ for both 
$\mathcal{F}$ and~$\mathcal{G}$. So there is nothing interesting happening.
For $i=0$, $\lambda$ is an isomorphism for $\mathcal{G}$.   But for $\mathcal{F}$ it is the inclusion of $j_!\mathbb{C}_U$ in $\mathbb{C}_Y$, which is not 
a surjection (in $\Shv_c Y$): we have $\mathcal{F}_0=0$ while 
$(j_*j^*\mathcal{F})_0=\mathbb{C}$.   
\end{example}

\subsection{Recollections concerning Perverse Sheaves}
\label{recs}
The definition of the class of objects in $\Dbc Y$ we are looking for
is tied up with the theory of perverse sheaves.  So we will
explain the part of that theory that we need here.  

For the rest of this subsection, we fix a scheme $Y$ which is reduced and
of finite type over $\mathbb{C}$.  We point out that everything we are
going to do goes on in $\Dbc Y$ and is essentially topological with respect
to $Y(\mathbb{C})$.  So we do not really need the
assumption that $Y$ is reduced: if $W$ is a scheme of finite type over
$\mathbb{C}$, then $\Dbc W$ is the same as $\Dbc W_{\red}$.  But it
makes it slightly more convenient to say certain things.  On the other
hand, while we do not require $Y$ to be separated, this is mostly to
conform to the setup of~\cite{bbd}.  In the end, the theorems we are
really interested in are local near a closed point in $Y$.  So we
could just as well assume that $Y$ is separated (or even affine).

We write $\Perv Y$ for the category of perverse sheaves on $Y$ (for
the middle perversity).  This is a full subcategory of the derived
category $\Dbc Y$.  For $K\in\Dbc Y$ and $i\in\mathbb{Z}$, we write
$\pH^iK$ for the $i$th perverse cohomology sheaf.  So, while $H^iK$ is
a usual sheaf on $Y$, $\pH^iK$ is an object in $\Perv Y$.  We have $\pH^i K[j]=\pH^{i+j}K$.

Suppose $j:V\hookrightarrow Y$ is a (locally closed) immersion of schemes,
and $K$ is a perverse sheaf on $V$.  Then we write
$j_{!*}K$ for the intermediate extension of $K$ to
$Y$, a perverse sheaf on $Y$ supported on the Zariski closure of $V$.
The intermediate extension is an extension of $K$ in the sense that
there is a natural isomorphism $j^*j_{!*}K=K$.   In other words, the restriction
of the intermediate extension to $V$ is just $K$.   In fact, we have the following
characterization of the intermediate extension.

\begin{theorem}[BBD]
\label{no-sub-quot}
  The intermediate extension $j_{!*}K$ is the unique extension of $K$
  in $\Perv Y$ supported on $\overline{V}$  with no nontrivial sub or quotient object supported
  on $\overline{V}\setminus V$.  
\end{theorem}

\begin{proof}
  This follows from~\cite[Corollaire 1.4.25]{bbd}. 
\end{proof}

\begin{proposition}
  \label{prop-van-comp}
  Suppose $V$ is a (locally closed) subscheme of $Y$ and
  $K$ is a perverse sheaf on $V$.   Then
 \begin{equation}
\label{e.van}
H^i(j_{!*}K)=0 \text{ for $i\notin [-d_V,0]$.} 
\end{equation}
Moreover, if we write  $j:V\hookrightarrow Y$
for the inclusion,  then 
\begin{equation}
\label{e.comp}
  H^{-d_V}(j_{!*}K)=j_*H^{-d_V}K.
\end{equation}
\end{proposition}

\begin{proof}
  Since $K$ is perverse, so is $j_{!*}K$, and it is supported on the
  Zariski closure $\overline{V}$ of $V$.  Therefore $H^ij_{!*}K=0$ for
  $i>0$ by \cite[Definition 2.1.2]{bbd}.  But, by the discussion in the
  two paragraphs just after Definition 2.1.2, $H^ij_{!*}K=0$ for $i<-d_V$.
  So~\eqref{e.van} is proved.  

  Since $j_{!*}K$ is an extension of
  $K$ supported on $\overline{V}$, $j^*H^{-d_V}(j_{!*}K)=H^{-d_V}K$.  So
  adjunction gives a natural morphism
  $H^{-d_V}(j_{!*}K)\to j_*H^{-d_V}K$.  Now~\eqref{e.comp} follows 
  easily from Deligne's formula \cite[Proposition 2.1.11]{bbd}, which
  computes $j_{!*}K$ in terms of a series of derived
  pushforwards and truncations.
\end{proof}

If $V$ is smooth and irreducible and $\mathcal{L}$ is a local system
on $V$, then $\mathcal{L}[d_V]$ is a perverse sheaf on $V$.  So the
intermediate extension $\IC\mathcal{L}:=j_{!*}\mathcal{L}[d_V]$ is a
perverse sheaf on $Y$, which is colloquially called the \emph{IC} sheaf or
the \emph{intersection cohomology} sheaf. 
For $d_V>0$, Deligne's formula actually implies that $H^i\IC(\mathcal{L})=0$ for $i\notin [-d_V,0)$.    So we get a slightly stronger vanishing
statement than~\eqref{e.van}.  
Note that, if $V'$ is a Zariski
dense, Zariski open subset  of $V$, then
$\IC(\mathcal{L}_{|V'})=\IC\mathcal{L}$, i.e., the two perverse
sheaves are canonically isomorphic~\cite[Lemme 4.3.2]{bbd}.

By~\cite[Theorem 4.3.1]{bbd}, a perverse sheaf $K$ on $Y$ is simple (as an object in the abelian category $\Perv Y$) if
and only if it is isomorphic to $\IC\mathcal{L}$ where $\mathcal{L}$
is an irreducible local system on a smooth, irreducible subscheme $V$ as
above.
So, a perverse sheaf $K$ is semisimple if and only if it is a direct
sum of such sheaves.   (Such a direct sum is necessarily finite because the
category of perverse sheaves is artinian.)   Suppose $Z$ is a closed subvariety 
of $Y$ (that is, $Z$ is an integral, closed subscheme).   
Following M.~Saito's notation 
from~\cite{SaitoKaehler}, we say that a perverse sheaf $K$ has \emph{strict
  support $Z$} if it is supported on $Z$ and has no proper sub or quotient
object supported on a proper subscheme of $Z$.   We write $\Perv_Z Y$
for the full subcategory of $\Perv Y$ consisting of perverse sheaves
with strict support~$Z$.   By Theorem~\ref{no-sub-quot}, if $\mathcal{L}$ is a local system on a non-empty, smooth, Zariski open subscheme 
$V$ of $Z$, then $\IC\mathcal{L}$ has strict support $Z$.   It follows
that any semisimple perverse sheaf $K$ on $Y$ can be written as a direct
sum
\begin{equation}
  \label{dsum}
  K=\oplus_Z K_Z
\end{equation}
where $Z$ ranges over all closed subvarieties of $Y$ and $K_Z\in \Perv_Z Y$
(with $K_Z=0$ for all but finitely many $Z$).
This decomposition is easily seen to be  unique (as there are no nonzero
morphisms between objects in $\Perv_ZY$ and $\Perv_{Z'}Y$ for $Z\neq Z'$).

Obviously, if $K$ is simple, then we must have $K_Z=0$ for all but one irreducible closed subscheme $Z$ of $Y$.   We call this subscheme the \emph{strict support of~$K$}.

\begin{lemma}
\label{suplem}
Suppose $K$ is a perverse sheaf on a scheme $Y$ of finite type over
$\mathbb{C}$, and let $j:V\hookrightarrow Y$ denote the inclusion of a
 Zariski dense, Zariski open subset.  Then $j^*K$ is perverse
on $V$.  If $K$ is simple with strict support equal to an irreducible
component of $Y$, then 
  \begin{equation}
\label{eresv}
  K=j_{!*}j^*K.
\end{equation}
\end{lemma}

\begin{proof}
  If $j$ is an open immersion (or even an \'etale morphism), then $j^*$
  always takes perverse sheaves to perverse
  sheaves~\cite[Corollaire 2.2.6 (ii)]{bbd}.  This proves the first assertion.

  Now suppose $K$ is simple with strict support equal to an
  irreducible component $Z$ of $Y$. Then  $K$ 
is an extension of $j^*K$ with no non-trivial sub or quotient object supported 
on $Z\setminus V$.    Since,  by Theorem~\ref{no-sub-quot}, $j_{!*}j^*K$ is the
  unique such extension, it follows that $K=j_{!*} j^*K$.  
\end{proof}

\begin{lemma}
  \label{l.obvious}
  Suppose $Y$ is a scheme of finite type over $\mathbb{C}$
  and let $\{\mathcal{F}\}_{i=1}^n$ be a finite collection of 
 sheaves on $Y$.  If $\mathcal{F}:=\oplus\mathcal{F}_i$ is a local system, then each $\mathcal{F}_i$ is as well. 
\end{lemma}
\begin{proof}
  Any idempotent $p\in \End\mathcal{F}$ is locally constant on
  $\mathcal{F}$. In particular, $p$ has locally constant rank.  So
  $\ker p$ and $\ker 1-p$ are local systems.  The result then follows
  by induction.
\end{proof}

\begin{definition}
  \label{d.mol}
  Suppose $U$ is a Zariski open subset of $Y$ and $K\in\Dbc Y$.
We say that $U$ is a \emph{mollifying} subset for $K$ if $U$
is Zariski dense in $Y$ and, for
all $i\in\mathbb{Z}$, $H^iK_{|U}$ is a local system on $U$. 
\end{definition}

\begin{lemma}
  \label{suplem2}
  Suppose $K\in\Dbc Y$.  Then $K$ has a mollifying subset.   In fact, we can 
even find one which is smooth.  
\end{lemma}

\begin{proof}
  This follows from generic smoothness and the definition of a
  constructible sheaf.
\end{proof}

\begin{remark}
  It turns out that we will not really need the existence of smooth 
  mollifying subsets, and in~\cite[Corollaire 6.2.9]{bbd},
  BBD do not use it. 
\end{remark}

\begin{proposition}
\label{prop-unstable}
  Suppose $K$ is a simple perverse sheaf on $Y$ with strict support $Z$, and 
  $j:U\hookrightarrow Y$ is the inclusion of a mollifying subset for $K$.   
\begin{enumerate}
\item If $Z$ is an irreducible component of $Y$, then 
$K_{|U}=\mathcal{M}[d_Z]$ where $\mathcal{M}:=H^{-d_Z}K_{|U}$, and 
$K=j_{!*}\mathcal{M}[d_Z]$. 
Moreover, $j_*j^*H^iK=0$ for $i\neq -d_Z$ and the local invariant
cycle map $\lambda_K^{-d_Z}$ is an isomorphism. 
\item Otherwise 
$j^*K=0$. Therefore $j_*j^*H^iK=0$ for all $i$.  
\end{enumerate}
\end{proposition}

\begin{proof}
  Suppose $V$ is an irreducible component of $U$.  If $Z$ does not contain
$V$, then, for each $i$, $H^iK_{|V}$ is a local system on $V$ supported on the  
closed, proper subscheme $Z\cap V$ of $V$.  Since $V$ is connected, it follows that $H^i K_{|V}=0$ for all~$i$.   This proves (2).   

So assume $Z$ is an irreducible component of $Y$ and let $V=Z\cap U$.
Since $U$ is dense in $Y$, $V$ is dense in $Z$.  So $d_V=d_Z$. 
We then have $H^iK=0$ for $i\notin [-d_Z,0]$ by
Proposition~\ref{prop-van-comp}.  On the other hand, if $i>-d_Z$, then 
$\dim\supp H^iK_{|V}\leq -i<d_Z$.   So again $H^iK_{|V}=0$ as it is a 
local system.     Therefore $H^iK_{|U}=0$ for all $i\neq -d_Z$, and it follows
that $K_{|U}=\mathcal{M}[d_Z]$ where $\mathcal{M}=H^{-d_Z}K_{|U}$. This shows
that $j_*j^*H^iK=0$ for $i\neq -d_Z$.  

 By Lemma~\ref{suplem}, we then get that $K=j_{!*}\mathcal{M}[d_Z]$.  
Then, by Proposition~\eqref{e.comp}, $H^{-d_Z}K=H^{-d_Z}(j_{!*}\mathcal{M}[d_Z])
=j_*H^{-d_Z}(\mathcal{M}[d_Z])=j_*j^*H^{-d_Z}K$.   So the local invariant 
cycle map $\lambda_K^{-d_Z}$ is an isomorphism.  
\end{proof}

\begin{corollary}
\label{cor-main-drive}
  Suppose $K$ is a simple perverse sheaf on $Y$, 
and $j:U\hookrightarrow Y$ is the inclusion
of a mollifying open subset for $K$. 
\begin{enumerate}
\item The local invariant cycle map $\lambda^i_K:H^iK\to j_{*}j^*H^iK$ is a 
surjection for all $i$.
\item Suppose every irreducible component of $Y$ has dimension $d_Y$
and suppose $i$ is an integer not equal to $-d_Y$.  Let $y\in Y(\mathbb{C})$
be a closed point of $Y$, and write $\lambda^i_K(y)$ for the map in 
~\eqref{e.lambday}.    Then $\lambda^i_K(y)$ is 
an isomorphism if and only if $H^iK_y=0$.
\item If $Y$ is equidimensional as in (2), then $\lambda^{-d_Y}_K$ is
  an isomorphism.  
\end{enumerate}

\end{corollary}
\begin{proof}
  Proposition~\ref{prop-unstable} shows that either $\lambda^i_K$ is
  an isomorphism or $j_*j^*H^iK=0$.  So (1) holds, because, in either
  case, $\lambda^i_K$ is a surjection.  Similarly, (3) holds because,
  in case (1) of Proposition~\ref{prop-unstable}, $\lambda_K^{-d_Z}=\lambda_K^{-d_Y}$
  was proven to be an isomorphism and, otherwise, $H^{-d_Y}K=0$ by
  Proposition~\ref{prop-van-comp}, which, by (1), trivially implies that 
$\lambda^{-d_Y}_K$ is an isomorphism.

For (2), suppose $i\neq -d_Y$.   Then
$j_*j^*H^iK=0$, again by Proposition~\ref{prop-unstable}.
Therefore, $\lambda^i_K(y)$ is an isomorphism if and only
if $H^iK_y=0$.

\end{proof}

\begin{corollary}
\label{c-ind}
  Suppose $K$ is a simple perverse sheaf on $Y$ and $U$ and 
  $V$ are two mollifying open subsets for $K$ with
  inclusions $j_{U}$ (resp.\ $j_V$) into $Y$.  Then the local invariant cycle
  maps for $U$ and $V$ are canonically isomorphic.   More precisely,
  $U\cap V$ is also a mollifying subset, and, if we let $j_{U\cap V}:U\cap V\hookrightarrow Y$ denote the inclusion, then, 
  for each $i\in\mathbb{Z}$, 
we have have a commutative diagram
$$
\xymatrix{
             &  j_{U*}j_U^*H^i K\ar[rd]^{\res} & \\
H^i K\ar[ur]^{\lambda^i_K}\ar[dr]_{\lambda^i_K}\ar[rr]^{\lambda_K^i}      &  & j_{(U\cap V)*}j^*_{U\cap V} H^iK \\
             &   j_{V*}j_V^*H^i K\ar[ru]^{\res}  & \\
}
$$
where 
$\res$ denotes restriction.   Moreover, each map labeled $\res$ is an
iso\-morphism. 
\end{corollary}
\begin{proof}
  It  is obvious that $U\cap V$ is a mollifying subset, and it is also
  very easy to see that the diagram above commutes.    So, 
  let $Z$ be the strict support of $K$.  If $Z$ is not an irreducible
  component of $Y$ or if $i\neq -d_Z$, then, by
  Proposition~\ref{prop-unstable}, the right three vertices
  are all $0$.  So there is nothing to prove.

  Otherwise, all of the arrows labeled $\lambda^i_K$ are isomorphisms.
  So the commutativity of the diagram shows that the maps labeled $\res$
  are all isomorphisms.  
\end{proof}

\begin{definition}
\label{def-ss}
  Suppose $Y$ is a scheme of finite type over $\mathbb{C}$.   An object
  $K\in\Dbc Y$ is said to be \emph{semisimple} if $K\cong\oplus_{i\in\mathbb{Z}} (\pH^iK)[-i]$ with each summand a semisimple perverse sheaf. 
\end{definition}

\begin{remark}
  If $K\in\Dbc Y$, then it is not hard to see (directly from the definitions
in~\cite{bbd}) that 
 $\pH^i K=0$ for all but finitely many $i\in\mathbb{Z}$.
\end{remark}

\begin{lemma}
\label{l.obv2}
  Suppose $K$ is a semisimple object in $\Dbc Y$ and $U$ is a Zariski dense,
Zariski open subset of $Y$.  Then $U$ is mollifying for $K$ if and only
if it is mollifying for every sub perverse sheaf of every perverse
cohomology sheaf $\pH^iK$. 
\end{lemma}

\begin{proof}
  Since the direct sum of local systems is a local system, it is obvious
that, if $U$ mollifies all the $\pH^iK$, it mollifies $K$ as well.  
This proves one direction of the assertion. The converse direction follows from
Lemma~\ref{l.obvious}. 
\end{proof}

\begin{theorem}
\label{main-lic-surj}
Suppose $K$ is a semisimple object in $\Dbc Y$ and let 
$U$ be mollifying for $K$. Then, for each $i$,
  the local invariant cycle map $\lambda^i_K$ is a surjection. 
\end{theorem}

\begin{proof}
  Using Lemma~\ref{l.obv2} along with the naturality of $\lambda$ explained 
in \S\ref{gen-def}, we can assume that $K$ is a simple perverse sheaf.  Then the result follows
  from Corollary~\ref{cor-main-drive} (1).  
\end{proof}

Theorem~\ref{main-lic-surj} and the decomposition theorem, \cite[Th\'eor\`eme 6.2.5]{bbd}, are the main 
ingredients in the proof of Theorem~\ref{thm:bbd}.  Analogously, 
the main ingredients of the proof of Theorem~\ref{t1} are the decomposition
theorem, the perverse hard Lefschetz theorem and the next result.

\begin{theorem}
  \label{t-good-cond}
  Suppose that $Y$ is equidimensional and $K$ 
 is a semisimple object in $\Dbc Y$.  Let 
  $y\in Y(\mathbb{C})$ be a closed point, and  
let $j:U\hookrightarrow Y$ be the inclusion of a mollifying subset for $K$.   Then the following are
  equivalent.
\begin{enumerate}
\item $\lambda(y):H^iK_y\to (j_*j^*H^iK)_y$ is an isomorphism for all $i$.
\item For all $j$ and all $i\neq -d_Y$, $H^i(\pH^jK)_y=0$. 
\end{enumerate}
\end{theorem}

\begin{proof} By shifting and passing to direct summands via~\ref{l.obv2},
we can assume that $K$ is a simple perverse sheaf.   
  The theorem is then a direct consequence of Corollary~\ref{cor-main-drive} (2) and (3).  
\end{proof}

\subsection{Proof of the Local Invariant Cycle Theorem}

Suppose now that $X$ is a reduced scheme of finite type
over $\mathbb{C}$ and $K\in\Dbc X$.  If $i:W\to X$ is the inclusion of
a subscheme, we write $\HH^j(W,K):=\HH^j(W, i^*K)$.  

Suppose $Y$ is another reduced scheme of finite type over
$\mathbb{C}$ and $y\in Y(\mathbb{C})$ is a closed point.  Then by a \emph{ball} centered at $y$, we mean any open neighborhood of $y$ in  $Y(\mathbb{C})$ which is obtained
by intersecting an affine open neighborhood $V$ of $y$ embedded in 
$\mathbb{A}^n_{\mathbb{C}}$ with a ball in $\mathbb{C}^n$.   

\begin{remark}
  We have to distinguish this notion of a ball from the notion of a ball
as in Theorem~\ref{t-dimca} which we use when $Y$ is smooth.  The issue
is that, if $Y$ is singular, then we can not necessarily find an open neighborhood of $y$ which is homeomorphic to an actual ball in $\mathbb{C}^{d_Y}$.  
\end{remark}

In the proof of the next theorem, we are going to use the Kashiwara conjecture.
As we mentioned above in \S\ref{ss-main}, this is now a theorem owing to the
work of many authors.   We will not cite these authors again here, but
we point out that Drinfeld's paper~\cite{drinfeld-kashiwara} 
is short and has a very efficient
statement of the part of the conjecture having to do with perverse sheaves
(which is the part that we~use).

\begin{theorem}
\label{big-lic} Suppose $f:X\to Y$ is a proper morphism of reduced  schemes of finite type over $\mathbb{C}$,  and let $K$ be a 
  semisimple object in $\Dbc X$.  Let $y\in X(\mathbb{C})$ be a closed
  point and set $X_y=f^{-1}(y)$.  Let $U\subset Y$ be a mollifying open subset
for the complex $Rf_*\mathbb{C}$.  
Then, for every sufficiently small ball $B(y)$ centered at~$y$, the local 
invariant cycle map 
induces a natural surjection 
\begin{equation}
\label{gen.map}
\rH^i(X_y,K)\twoheadrightarrow \rH^0(B(y)\cap U, R^if_*K).
\end{equation}
Moreover, the target of \eqref{gen.map} is independent of the mollifying
open subset $U$.  
\end{theorem}

\begin{proof}
  First note that, since the whole theorem is a Zariski local
  statement near~$y$, we can easily reduce to the situation where $Y$
  (and, therefore, $X$) are separated.  For example, we can replace
  $Y$ with an affine Zariski open neighborhood of~$y$ and $X$ with the
  inverse image of that neighborhood.  So we assume $X$ and~$Y$ are
  separated.  (We do this in order to freely use work on Kashiwara's
  conjecture.)

  The source of \eqref{gen.map} is naturally isomorphic to $H^iK_y$ by
  proper base change.  The target is naturally isomorphic to
  $(j_*j^*H^iK)_y$ for $B(y)$ sufficiently small.  Since $K$ is
  semisimple, it follows from the Kashiwara conjecture that $Rf_*K$ is semisimple.  Therefore the surjectivity
  result follows from Theorem~\ref{main-lic-surj}.

The independence of $U$ follows from Corollary~\ref{c-ind}.  
\end{proof}

\begin{proof}[Proof of Theorem~\ref{thm:bbd}]
  Since $X$ is smooth, the constant sheaf $\mathbb{C}=\mathbb{C}_X$ is
  semisimple. In fact, we even assumed that $X$ is irreducible, so 
  $\mathbb{C}[d_X]$ is the IC sheaf of the simple local systems
  $\mathbb{C}_X$.  Therefore $\mathbb{C}[d_X]$ is a simple perverse sheaf.
  Since $f$ is smooth and proper over $U$, $U$ is a mollifying subset for
  $Rf_*K$.  
  So take $K=\mathbb{C}$ in Theorem~\ref{big-lic}.
  This proves that~\eqref{e.lic} of Theorem~\ref{thm:bbd} holds.

  Since a ball $B(y)$ as in Theorem~\ref{t-dimca} is homeomorphic to
  an open ball in $\mathbb{C}^{d_Y}$, $B(y)\cap U$ is non-empty and
  connected.  This follows from Fact~\ref{lmfact} in the case $d_Y>0$
  and it is obvious otherwise.  Then~\eqref{e.lic-pi1} follows from
  proper base change and Corollary~\ref{c.equiv}.
\end{proof}

\begin{remark}
  In~\cite{bbd}, BBD prove Theorem~\ref{big-lic} for $K$ semisimple of
  geometric origin.  (These are essentially the complexes that can be
  obtained from the constant sheaf by the standard operations of sheaf
  theory, e.g., Grothendieck's six operations.)  The reason for the restriction
was that they were only able to prove the decomposition theorem~\cite[Th\'eor\`eme 6.2.5]{bbd} for such complexes.   

BBD also state Theorem~\ref{big-lic} for arbitrary schemes of finite type (without the restriction that $X$ or $Y$ be reduced).  But, as we mentioned at the
beginning of \S\ref{recs}, there is actually no loss in generality in assuming that the schemes involved are reduced.  
\end{remark}

\subsection{The Palindromicity Theorem}

It will be convenient to introduce some notation concerning
palindromic polynomials.

\begin{definition}
  Suppose $p\in \mathbb{R}[t,t^{-1}]$ is a Laurent polynomial.  We say that 
$p$ is \emph{palindromic} if $p(t)=p(t^{-1})$. 
\end{definition}

\begin{lemma}\label{palder}
 Suppose $p\in\mathbb{R}[t,t^{-1}]$ is palindromic.  Then $p'(1)=0$.
\end{lemma}

\begin{proof}
  By palindromicity and the  chain rule,  $$p'(t)=\frac{d}{dt} p(t^{-1})=-p'(t^{-1})t^{-2}.$$  So $p'(1)=-p'(1)$.  Therefore, $p'(1)=0$.
\end{proof}

If $p$ is any Laurent polynomial, we think of $p'(1)$ as the \emph{center
  of mass} of $p$.

\begin{lemma}\label{l1}
  Suppose $q=\sum_{\ell\geq 0} p_{\ell}t^{\ell}$ where the 
  $p_{\ell}$ are palindromic Laurent polynomials with real
  coefficients.  Assume that, for $\ell>0$, the coefficients of the
  $p_{\ell}$ are non-negative.  Then
  \begin{equation}
    \label{qpal}
    q'(1)\geq 0.
  \end{equation}
  Moreover, the following are equivalent:
  \begin{enumerate}
  \item We have equality in \eqref{qpal}.
  \item $p_{\ell}=0$ for all $\ell>0$.
  \item $q$ is palindromic.
  \end{enumerate}

\end{lemma}

\begin{proof}
  Since the $p_{\ell}$ are all palindromic, we have $p'_{\ell}(1)=0$ for
  all $\ell\geq 0$.  Therefore
  \begin{equation}
  q'(1)=\sum_{\ell>0} \ell p_{\ell}(1).\label{qder}
  \end{equation}
  Since the coefficients of the $p_{\ell}$ are non-negative for $\ell>0$,
  \eqref{qder} implies \eqref{qpal}.  From \eqref{qder}, it is also immediate that
  (1) implies (2).   Then (2)$\Rightarrow$ (3) is obvious, and (3)$\Rightarrow$ (1) follows from Lemma~\ref{palder}.  
\end{proof}

\begin{theorem}
  \label{jbpal}
  Suppose $Y$ is a scheme of finite type over $\mathbb{C}$ and
  $K\in\Dbc Y$ is a semisimple complex.   Assume that 
\begin{equation}
    \label{lefis}
    \pH^{-i}K\cong\pH^{i}K
  \end{equation}
  for all $i\in\mathbb{Z}$.
  Fix $y\in Y(\mathbb{C})$ and $r\in\mathbb{Z}$, and
  suppose that
  \begin{equation}
    \label{jbhyp}
    H^j(\pH^iK)_y=0\text{ for $i,j\in\mathbb{Z}$ such that $j<r$.}
  \end{equation}
   Set
  \begin{equation}
    \label{eqt}
    q(t):=\sum_{k\in\mathbb{Z}} t^k\dim H^k(K[r])_y.
    \end{equation}
    Then $q'(1)\geq 0$.  Moreover, the following are equivalent:
    \begin{enumerate}
    \item $q'(1)=0$.
    \item We have 
    \begin{equation}
      \label{nohi}
      H^j(\pH^i K)_y=0\text{ for $j\neq r$.}
    \end{equation}
  \item $q$ is palindromic.
    \end{enumerate}
\end{theorem}
\begin{proof}
  Since $K$ is semisimple, we have
  \begin{equation}
    \label{ssgeo}
    K=\oplus (\pH^iK)[-i].
  \end{equation}
  Here we sum over $i\in\mathbb{Z}$, but,
  since this is clear from the context, we leave this out of summation
  notation to save clutter (as we will do in the rest of the proof as well).

  Substituting in \eqref{ssgeo} for $K$ in \eqref{eqt}, we get 
  \begin{align}
    q(t)&=\sum t^{k}\dim H^k(K[r])_y\nonumber\\
        &=\sum_{k\in\mathbb{Z}} t^k\sum_{i\in\mathbb{Z}} \dim H^k\left((\pH^iK)[r-i]\right)_y\nonumber\\
        &=\sum t^{k} \dim H^{k+r-i}(\pH^iK)_y\label{bnumber}\\
        &=\sum_{\ell,i\in\mathbb{Z}} t^{i+\ell} \dim H^{\ell+r}(\pH^iK)_y,\label{expnumber}     \\
    &=\sum_{\ell} t^{\ell}\sum_{i} t^i\dim H^{\ell+r}(\pH^iK)_y. 
  \end{align}
  where here we go from \eqref{bnumber} to \eqref{expnumber} by substituting
  $k=i+\ell$.  

  Now, for $\ell\in\mathbb{Z}$, set 
  \begin{equation}
    \label{epl}
    p_{\ell}:=\sum_{i\in\mathbb{Z}} t^i \dim H^{\ell+r}(\pH^iK)_y.
  \end{equation}
  By~\eqref{jbhyp}, $p_{\ell}=0$ for $\ell<0$. So, by~\ref{expnumber},
  $q=\sum_{\ell\geq 0} t^{\ell}p_{\ell}$.  But, by~\eqref{lefis}, each $p_{\ell}$
  is palindromic.  The theorem now follows from Lemma~\ref{l1}.  
\end{proof}

\begin{theorem}
  \label{t-pal-main}
  Suppose $f:X\arr Y$ is a projective morphism of reduced schemes of finite type over
  $\mathbb{C}$ with $Y$ equidimensional.  Let $y\in Y(\mathbb{C})$
  be a closed point and let $K$ be a semisimple complex
  in  $\Perv X$.  Set 
  \begin{equation}
    \label{e.Ktwo}
    q(t):=\sum_{k\in\mathbb{Z}} t^k  \dim \rH^{k-d_Y}(X_y,K).
  \end{equation}
  Then $q'(1)\geq 0$.   Moreover, the following are equivalent:
  \begin{enumerate}
  \item $q'(1)=0$.
  \item The local invariant cycle maps in~\eqref{gen.map} are isomorphisms
    for all $i$.
  \item $q$ is palindromic.  
  \end{enumerate}
\end{theorem}
\begin{proof}
  As in the proof of Theorem~\ref{big-lic}, we can (and do) assume
  that $Y$ is separated.  Then, by Kashiwara's conjecture, $Rf_*K$ is
  a semisimple complex and, for each $i$, we have an isomorphism
  $\pH^{-i}Rf_*K\cong \pH^iRf_*K$.  (This is the ``Hard Lefschetz''
  part of the Kashiwara conjecture, Item 2 in Drinfeld's
  statement~\cite{drinfeld-kashiwara}.) On the other hand, we have a
  canonical isomorphism
  $\rH^{k-d_Y}(X_y,K)=\rH^k(X_y,Rf_*K[-d_Y])=H^k(Rf_*K[-d_Y])_y$.
  Finally, for any complex $C$ on $Y$, we have $H^j(\pH^i C)=0$ for
  $j<-d_Y$ by Proposition~\ref{prop-van-comp}.  So, in particular,
  $H^j(\pH^i Rf_*K)_y=0$ for $j<-d_Y$.

  Now,  by Theorem~\ref{t-good-cond}
  we have   $H^j(\pH^iRf_*K)_y=0$ for all $j\neq -d_Y$ for all $j$
  if and only if the local invariant cycle maps are all isomorphisms. 
 So the theorem now follows from Theorem~\ref{jbpal}.  
\end{proof}

\begin{proof}[Proof of Theorem~\ref{t1}] Since $X$ is smooth and
  irreducible, $\mathbb{C}[d_X]$ is semisimple perverse.   So let
  $K=\mathbb{C}[d_X]$ in Theorem~\ref{t-pal-main}.
  Set 
$$q(t)=\sum_{i\in\mathbb{Z}} t^i\dim \rH^i(X_y,\mathbb{C}[d])=
  \sum_{i\in\mathbb{Z}} t^i\dim \rH^i(X_y,\mathbb{C}[d_x-d_Y]).$$
  Then $q$ is palindromic if and only if \eqref{e.t1} holds.
  So the result follows from Theorem~\ref{t-pal-main}.  
\end{proof}

\section{Galois covers}\label{gc}
The purpose of this section is to prove a lemma about the local
monodromy groups of Galois covers.  We use the concepts of
\S\ref{lmlfg}, but we have changed some of the notation (partially to avoid running out of capital letters
towards the end of the alphabet).

\subsection{Covers and monodromy}
Suppose $U$ is a smooth, connected, complex, quasi-projective variety
and $G$ is a finite group acting freely on $U$.  Let $V=U/G$, and write
$\pi:U\to V$ for the quotient morphism.
(It is well known that $U/G$ is a scheme. See, for example, 
\cite[\href{http://stacks.math.columbia.edu/tag/0725}{Tag 0725}]{stacks-project}.)

For each (closed) point $u\in U$, we
get a surjective group homomorphism
\begin{equation}
  \label{e.GaloisHomomorhpism}
  \psi_u: \pi_1(V,\pi(u))\twoheadrightarrow G.
\end{equation}
If $\gamma:[0,1]\to V$ represents an element of $\pi_1(V,\pi(u))$ and
$\tilde{\gamma}$ is a lift of $\gamma$ to $U$ with
$\tilde{\gamma}(0)=u$, then $\psi_u(\gamma)u=\tilde{\gamma}(1)$.  From
this description, we see that, if $\pi(u')=\pi(u)$, then $\psi_u$ and
$\psi_{u'}$ differ by conjugation by an element of $\pi_1(V,\pi(u))$.
(See~\cite{Hatcher} for a complete discussion of these matters.)

Now, suppose that $V$ is contained as a Zariski open subset of a smooth,
quasi-projective variety $Y$.  Set $Z=Y\setminus V$, and suppose $z$ is a 
closed point of~$Z$.  Let  $W$ be a good neighborhood of $z$ in $Y$ relative
to $Z$, and let $w$ be a point in $W\setminus Z$.  The choice of a path from 
$w$ to $\pi(u)$ gives us a 
sequence of group homomorphisms 
\begin{equation}\label{lmg}
\pi_1(W\setminus Z,w)\to \pi_1(V,\pi(u))\twoheadrightarrow G.
\end{equation}
Moreover, up to conjugation by an element of $G$, this map is
independent of $u$, the path from $w$ to $\pi(u)$, $W$ and $w$.  We
call the image $H(z)$ of the composition in~\eqref{lmg} the \emph{local
  monodromy subgroup at $z$}.  (The conjugacy class of $H(z)$ is independent
of any choices.)   Note that, if we replace $V$ with any
non-empty Zariski open subset $V'$ of $V$ containing $\pi(u)$ and we
replace $U$ with $\pi^{-1}(V')$, then $H(z)$ does not change.  This follows from 
Lemma~\ref{lsurj}.

\begin{proposition}\label{pgal}
  Let $G$ be a finite group acting on a smooth, quasi-projective
  variety $X$, and suppose $G$ acts freely on a Zariski dense open
  subset $U$ of $X$.  Suppose $\pi:X\to Y$ is the quotient of $X$ by
  $G$ and let $V=\pi(U)$.  Pick a closed point $x\in X\setminus U$, and suppose that $Y$
  is smooth.  Then $H(\pi(x))$ is the stabilizer $G_x$ of the point
  $x$.
\end{proposition}

\begin{proof}
  Take a good neighborhood $B$ of $y:=\pi(x)$ with respect to $Z:=Y\setminus V$.
Pick $b\in B\cap V$ and set $A=\pi^{-1}B$.  Let $A_x$ denote the component of $A$ containing~$x$.   There exists $a\in A_x$ such that $\pi(a)=b$.  Let $H$ denote
the image of the composition $\pi_1(B\cap V,b)\to \pi_1(V,b)\twoheadrightarrow G$, where the last homomorphism is~$\psi_a$.
Then $H=H(a)$.   The group $G$ acts transitively
on the connected components of $A\cap U$, and the stabilizer of the
component of $A\cap U$ containing $a$ is $H$.  Since $A_x\cap U$ is connected, 
$A_x\cap U$ is this component.   So the stabilizer in $G$ of this component
is the same as the stabilizer of the $A_x$.  But, by possibly shrinking~$B$, we can arrange that this is just~$G_x$.  
\end{proof}

\begin{remark}
We use the assumption that $Y$ is smooth because we have only 
defined the local fundamental group in that case.   
However, since $Y$ is a quotient of a smooth variety, it is automatically
normal.  And good neighborhoods of normal quasi-projective varieties are
connected.  (See Chapter 3 of Mumford's~\cite{mumford-cpv}). It follows that the assumption that $Y$ is smooth 
can be dropped. 
\end{remark}

\section{Geometry of Hessenberg Schemes}\label{sghs}

In this section, we study the geometry of the family of Hessenberg
varieties over the space of regular matrices.  We also study the
family of maximal tori defined by centralizers of regular, semisimple
matrices.  Ng{\^o}'s paper on the Hitchin fibration~\cite{NgoHitchin}
significantly influenced our thinking about these matters, and we have consequently
borrowed Ng{\^o}'s notation.

\subsection{Regular matrices} Fix a positive integer $n$ and write
$\mathfrak{g}$ for the Lie algebra $\mathfrak{gl}_n$.  Recall that a
matrix $s\in\mathfrak{g}$ is regular if the Jordan blocks of $s$ have
distinct eigenvalues.  A matrix $s$ is regular if and only if its centralizer
has dimension $n$.  As in \S\ref{intro}, we say that $s$ is regular of type
$\lambda$ for a partition $\lambda$ of $n$ if the Jordan blocks of $s$
are of sizes $\lambda_1,\ldots \lambda_r$.  We write
$\mathfrak{g}^{\mathrm{r}}$ for the subset of regular matrices and
$\mathfrak{g}^{\mathrm{r}}_{\lambda}$ for the subset of regular
matrices of type $\lambda$.  We write $\grs$ for the 
subset of regular semisimple matrices.   This is a dense open 
subset of $\mathfrak{g}$. 

\subsection{Hessenberg schemes}\label{hs} Fix a Hessenberg function
$\mathbf{m}=(m_1,\ldots, m_{n-1})$ with $\mathbf{m}(i):=m_i$, and 
set $m_n=n$. 
 Write
$\mathscr{X}$ for the variety of complete flags in $\mathbb{C}^n$, and set
$$
\mbox{$\sH(\bm) :=\{(F,s)\in\mathscr{X}\times\mathfrak{g}: 
sF_i \subseteq F_{m_i}$ for $1\leq i\leq n\}$.}
$$
Note that the projection $\pr_1$ on the first factor makes $\sH(\mathbf{m})$
into a vector bundle of rank $\sum_{i=1}^n m_i$ over $\sX$.  So $\sH(\mathbf{m})$ is a smooth,
connected scheme
with 
\begin{equation}\label{dimhm}
  \dim\sH(\mathbf{m})=\dim\mathscr{X}+\sum_{i=1}^n m_i  =\frac{n(n-1)}{2}+\sum_{i=1}^n m_i.
\end{equation}

Let $\pi:\sH(\mathbf{m})\to \mathfrak{g}$ denote
the projection on the second factor.  Then the fiber of $\pi$ over a
matrix $s\in\mathfrak{g}$ is simply the Hessenberg variety
$\sH(\mathbf{m},s)$.  Note that $\pi$ is smooth over
$\grs$.

\begin{theorem}\label{t.flat}
  The map $\pi:\sH(\mathbf{m})\to \mathfrak{g}$ is flat over the locus $\gr$ of
  regular matrices.
\end{theorem}

\begin{proof}
  Both $\sH(\mathbf{m})$ and $\mathfrak{g}$ are smooth over $\mathbb{C}$, and,
  by Corollary~\ref{dimcor}, 
  all fibers of $\pi$ over $\gr$ have the same dimension, $|\mathbf{m}|$.
  It follows from the theorem which is sometimes called ``miracle flatness'' that the restriction
  of $\pi$ to the inverse image of $\gr$ is flat. (See~\cite[Ex. III.10.9]{hartshorne} for miracle flatness.)
\end{proof}

\subsection{Diagonal matrices and characteristics}  Write
$G:=\mathbf{GL}_n$ and write $\mathbf{D}$ for the diagonal subgroup of
$G$.  Write $\mathbf{d}$ for the Lie algebra of $\mathbf{D}$, and
$\mathbf{d}_r$ for the regular elements of $\mathbf{d}$. The symmetric
group $S_n$ acts on $\mathbf{d}=\mathbb{A}^n$ in the obvious way:
$\sigma(x_1,\ldots, x_n)=(x_{\sigma^{-1}(1)},\ldots,
x_{\sigma^{-1}(n)})$.  The quotient is the characteristic variety
$\car=\car_G=\mathbf{d}/S_n$.  We can view $\car$ as the
variety of monic polynomials $t^n+a_{n-1}t^{n-1}+\cdots a_0$ of
degree $n$.  The \emph{Chevalley morphism}   
$\chi:\mathfrak{g}\to \car$ is the morphism sending a matrix
$s\in\mathfrak{g}$ to its characteristic polynomial $\chi(s)$.

\subsection{The smallest Hessenberg scheme}\label{smallest}
Let $\Bell$ denote the Hessenberg function $\Bell(i)=i$.
Then the restriction $\sH^{\mathrm{rs}}(\Bell)\to \grs$ of
$\pi:\sH(\Bell)\to \mathfrak{g}$ to the inverse image of $\grs$
is an \'etale cover of degree~$n!$.  Since $\sH(\Bell)$ is
connected and smooth, so is $\sH^{\mathrm{rs}}(\Bell)$.  So,
since $\grs$ is also connected, $\sH^{\mathrm{rs}}(\Bell)\to
\grs$ is an \'etale cover corresponding to an index $n!$ subgroup of
the fundamental group of $\grs$.

We call $\sH(\Bell)$ 
the smallest Hessenberg scheme because, for any Hessenberg function
$\mathbf{m}$, there is a canonical inclusion $\sH(\Bell)\to
\sH(\mathbf{m})$ (which is a closed immersion).  

We have a morphism $\tilde\chi:\sH(\Bell)\to \mathbf{d}$
sending a pair $(s,F)$ to the diagonal matrix with $\Gr^F_i s$ in the
$(i,i)$-entry.  This gives rise to a commutative diagram
\begin{equation}\label{gsr}
\xymatrix{
\sH(\Bell)\ar[r]^{\tilde\chi}\ar[d]^{\pi}   & \mathbf{d}\ar[d]^{\chi_{|\mathbf{d}}}\\
\mathfrak{g}\ar[r]^{\chi}             & \car,
}
\end{equation}
which coincides with Grothendieck's simultaneous resolution
of $\chi$. (For a discussion of Grothendieck's resolution, see Springer~\cite[\S 4.1]{springer} or Slodowy~\cite[\S 4.7]{slodowy}.) 
The restriction of ~\eqref{gsr} to the inverse image of
$\car^{\mathrm{rs}}:=\dr/S_n$ is a pullback
diagram.  In other words,
$\sH^{\mathrm{rs}}(\Bell)=\grs\times_{\car^{\mathrm{rs}}}
\dr$.  This shows that
$\sH^{\mathrm{rs}}(\Bell)$ is a (connected) Galois cover of
$\grs$ with Galois group $S_n$.

We can describe this Galois covering a little bit more explicitly if
we introduce the closed subscheme $Z$ of
$(\mathbb{P}^{n-1})^n\times\grs$ consisting of ordered tuples
$([v_1],\ldots, [v_n]; y)$ where the $v_i$ form a basis of
eigenvectors of $y$.  Given a point $z=([v_1],\ldots, [v_n]; y)$ in
$Z$, we can define a complete flag $F(z)$ by setting $F_i=\langle
v_1,\ldots, v_i\rangle$.  This defines a morphisms $Z\to
\sH^{\mathrm{rs}}(\Bell)$ given by $z\mapsto (F(z),y)$.  Using the fact
that $Z$ and $\sH^{\mathrm{rs}}(\Bell)$ are both \'etale covers of
$\grs$ of the same degree, it is easy to see that $Z\to
\sH^{\mathrm{rs}}(\Bell)$ is an isomorphism.  Then $S_n$ acts on $Z$ by
permuting the $v_i$: $\sigma([v_1],\ldots, [v_n];
y)=([v_{\sigma^{-1}(1)}],\ldots , [v_{\sigma^{-1}(n)}], y)$.  It is
easy to see that
the map $\tilde{\chi}:\sH^{\mathrm{rs}}(\Bell)\to \dr$
is $S_n$-equivariant.

\subsection{The fundamental groups}  
Suppose $z=(F,y)\in\sH^{\mathrm{rs}}(\Bell)$.
We get a surjection
$\psi_z:\pi_1(\grs,y)\twoheadrightarrow S_n$ corresponding
to the Galois covering $\sH^{\mathrm{rs}}(\Bell)\to \grs$ with Galois group $S_n$.   Similarly, for a regular diagonal matrix $u\in\mathbf{d}_r$, we have a surjection $\psi_u:\pi_1(\car^{\mathrm{rs}},\chi(u))\to S_n$.
Since
$$
\xymatrix{
\sH^{\mathrm{rs}}(\Bell)\ar[r]^{\tilde\chi}\ar[d] & \dr\ar[d]^{\chi_{|\dr}}\\
\grs\ar[r]^{\chi}             & \car^{\mathrm{rs}}.
}
$$
is a pullback diagram of Galois \'etale covers with Galois group $S_n$,
we have $\psi_z=\psi_{\tilde{\chi}(z)}\circ\chi_*$.

\begin{definition}
  We say a polynomial $p\in\car$ is of type $\lambda$ if $p=\prod_{i=1}^{\ell} 
(x-x_i)^{\lambda_i}$ where $x_1,\ldots , x_{\ell}$ are distinct.  
\end{definition}

\begin{lemma}\label{myoung}
  Suppose
  $p=\prod_{i=1}^{\ell} (x-x_i)^{\lambda_i}$
is a polynomial of type $\lambda$.
Then the local monodromy subgroup $H(p)$ of $S_n$ at $p$ for the $S_n$-cover
$\mathbf{d}_{\mathbf{r}}\to\carss$ is $S_{\lambda}$. 
\end{lemma}

\begin{proof}
Let $\tau$ denote the diagonal matrix
\[\diag(x_1,\ldots, x_1,x_2\ldots, x_2,\ldots, x_r,\ldots, x_r).\]
Then the stabilizer in $S_n$ of $\tau$ is precisely
$S_{\lambda}$.   The results then follows from Proposition~\ref{pgal}.
\end{proof}

\subsection{The Kostant section} The Kostant section is a morphism
$\epsilon:\car\to \mathfrak{g}^{\mathrm{r}}$ which is a section of $\chi$;
i.e., $\chi\circ\epsilon=\id$. 
 We give the definition of $\epsilon$
following Ng\^o's paper~\cite[Theorem 2.1]{NgoHitchin}.
We remark, however, that, while the general definition makes sense for
any reductive Lie algebra, we only discuss it for $\mathfrak{gl}_n$.  

Let $\mathbf{x}_{-}$ (resp.\ $\mathbf{x}_{+}$) denote the $n\times n$
matrix with $1$'s just below (resp.\ just above) the diagonal and $0$'s
everywhere else.  Then let $\mathfrak{g}^{\mathbf{x}_+}$ denote the
centralizer of $\mathbf{x}_+$ in $\mathfrak{g}$.  In~\cite{Kostant63},
Kostant showed that the subspace
$\mathbf{x}_{-}+\mathfrak{g}^{\mathbf{x}_{+}}$ is contained in
$\mathfrak{g}^{\mathrm{r}}$.  Moreover, he showed that the restriction of $\chi$
to $\mathbf{x}_{-}+\mathfrak{g}^{\mathbf{x}_+}$ induces an isomorphism
onto $\car$.  The Kostant section is the inverse morphism $\epsilon:
\car\to\mathbf{x}_{-}+\mathfrak{g}^{\mathbf{x}_+}$.  In the case of
$\mathfrak{gl}_n$,
\[
\mathbf{x}_{-}+\mathfrak{g}^{\mathbf{x}_+}=
\left\{
\begin{pmatrix}
x_0 & x_1 & x_2 & \ldots & x_{n-2} & x_{n-1} \\
1 & x_0   & x_1 & \ldots &  x_{n-3}& x_{n-2} \\
0 & 1   & x_0   & \ldots & x_{n-4} & x_{n-3} \\
\vdots & \vdots & \vdots &\ddots &\vdots & \vdots \\
0      & \ldots &           & \ldots & 1 & x_0\\
\end{pmatrix}\right\}.
\]
From this, it is elementary to compute the Kostant section. For example, 
for $n=2$, it sends the characteristic polynomial $p=x^2+a_1x+a_0$ to 
the matrix of the form above with $x_0=-a_1/2$, $x_1=a_1^2/4-a_0$.

\begin{proposition}\label{ryoung}
  Suppose $s\in\mathfrak{g}^r$ is a regular matrix of type $\lambda$.
Then the local monodromy $H(s)$ at $s$ for the $S_n$-cover 
$\sHrs(\Bell)\to\grs$  
is conjugate to the Young subgroup $S_{\lambda}$.  
\end{proposition}

\begin{proof} 
  We can assume that $s=\epsilon(p)$ for some $p\in\car$.  Then, by
  Proposition~\ref{psec}, the local fundamental group at $p$ is a
  retract of the local fundamental group at~$s$.  Since the
  $S_n$-cover $\sHrs(\Bell)\to\grs$ is a pullback of the $S_n$-cover
  $\mathbf{d}^{r}\to\carss$, it follows that the local monodromy
  subgroup at $s$ is equal to the local monodromy subgroup at $p$.  By
  Lemma~\ref{myoung}, this subgroup is $S_{\lambda}$.
\end{proof}

\begin{remark}
We could have used any (continuous) section of the map $\chi:\mathfrak{g}\to \car$ to prove Proposition~\ref{ryoung}.
\end{remark}

\subsection{The commuting group scheme}\label{cgs}
Write $I:=\{(g,x)\in G\times\mathfrak{g}: \Ad g (x)=x\}$.  The projection
$p:I\to\mathfrak{g}$ is a group scheme in a more or less obvious way.  Write $p^{\mathrm{rs}}:\mathscr{T}\to
\grs$ for the restriction of~$p$ to the inverse image of
$\grs$.  Then $\mathscr{T}$ is a torus bundle: the fiber over a point
$y\in\grs$ is the maximal torus in $G$ centralizing $y$.

Identify the scheme $Z$ from \S\ref{smallest} with $\sHrs(\Bell)$ and
form a pullback diagram
$$
\xymatrix{
  \mathscr{T}_Z\ar[r]\ar[d] & \mathscr{T}\ar[d]^{p^{\mathrm{rs}}}\\
       Z\ar[r]^{\pi}   & \grs.\\
}
$$
Then $\mathscr{T}_Z$ is equipped with an isomorphism
$\mathscr{T}_Z\to \mathbb{G}_{mZ}^{n}$ to the split torus over
$Z$.  To see this, suppose $z=([v_1],\ldots, [v_n];y)$ and $g\in G$ is
an element commuting with $y$.  Then $g$ preserves the eigenspaces of
$y$.  So for each $i=1,\ldots n$, there is a unique character
$t_i\in X^*(\mathscr{T}_Z)$ such that $gv_i=t_i(g)v_i$.  The 
$n$-tuple of characters $\vec{t}:=(t_1,\ldots, t_n):\mathscr{T}_Z\to \mathbb{G}_{m,Z}^{n}$ is easily seen to give an isomorphism. 

  Over $\mathfrak{g}^{rs}$, the torus $\mathscr{T}$ is determined
  up to isomorphism by its group of characters $X^*(\mathscr{T})$
  viewed as a $\mathbb{Z}$-local system over $\grs$.
  Moreover, this local system is canonically isomorphic to
  $R^1p^{\mathrm{rs}}_*\mathbb{Z}$.  For any point
  $y\in\grs$, the fundamental group
  $\pi_1(\grs,y)$ acts on the fiber of
  $X^*(\mathscr{T})$ lying over $y$ by permuting the characters
  $t_1,\ldots, t_n$.

\section{Monodromy and Tymoczko's dot action}\label{mtd}

\subsection{Fiberwise cohomology of $B\mathscr{T}$}
\label{fcbt}  

For each $y\in\grs(\mathbb{C})$, we have a torus $\mathscr{T}_y$ and
its associated classifying space $B\mathscr{T}_y$.  The cohomology of
$B\mathscr{T}_y$ is naturally a polynomial ring
$\mathbb{C}[t_1,\ldots, t_n]=
\mathbb{C}[X^*(\mathscr{T}_y)]$ generated in degree $2$ by the
characters of $\mathscr{T}_y$.  As we vary $y$, these glue together to
form a local system $\mathscr{A}$ of polynomial algebras
over $\grs$.  In fact, since $\mathscr{T}$ is \'etale locally trivial,
we can
construct a fiber bundle $a:B\mathscr{T}\to \grs$ over $\grs$ such that 
the fiber over each $y\in\grs$ is $B\mathscr{T}_y$.  
Then we have $\mathscr{A}=\oplus_{k\geq 0}\mathscr{A}^k$ 
where $\mathscr{A}^k=R^{2k}a_*\mathbb{C}$.

Let $y_0$ denote the regular
semisimple matrices with diagonal entries $1$, $2$, $3,\ldots, n$
(written in order).  Let $F^0$ denote the standard flag
$$\langle e_1\rangle \subset \langle e_1,e_2\rangle \subset \cdots$$
where $e_i$ is the standard basis of $\mathbb{C}^n$.  Set $z_0=(F^0,y_0)\in
\sH(\Bell,y_0)$.  This gives rise to a surjection
\begin{equation}\label{psifixed}
\psi:\pi_1(\grs,y_0)\to S_n,
\end{equation}
where, for simplicity, we write $\psi:=\psi_{z_0}$.  

Let $T$ denote the fiber of $\mathscr{T}$ over $y_0$.  So $T$ is simply
the diagonal subgroup of $G$.  Then, by the discussion in \S\ref{cgs}, $\pi_1(\grs,y_0)$ acts on $X^*(T)$ 
by permuting the characters.   Explicitly, if we let $\sigma(t_i)=
t_{\sigma^{-1}(i)}$ for $\sigma\in S_n$, then $\gamma(t_i)=
(\psi(\gamma))(t_i)$.  
Consequently, if we let $A=\mathscr{A}_{y_0}=\mathbb{C}[t_1,\ldots, t_n]$, then $\pi_1(\grs,y_0)$ acts on the polynomials in $A$ by 
$\gamma (p)=(\psi(\gamma))(p)$, where $S_n$ acts on $A$ in the standard way:
\begin{equation}\label{fcbtf}
(\sigma p)(t_1,\ldots, t_n)=p(t_{\sigma(1)},\ldots, t_{\sigma(n)}).
\end{equation}

\subsection{Fiberwise equivariant cohomology of Hessenberg varieties} Now, for each Hessenberg
function $\mathbf{m}$, the torus $\mathscr{T}$ acts on the Hessenberg
scheme $\sHrs(\mathbf{m})\to \grs$.  So for each $y\in\grs$, we can
take the equivariant cohomology groups
$\HH^*_{\mathscr{T}_y}(\sH(\mathbf{m}),y)$ (with complex coefficients).  By localization, we know
that $\HH^*_{\mathscr{T}_y}(\sH(\mathbf{m}),y))$ is a free module of
rank $n!$ over $\mathscr{A}_y= \HH^*(B\mathscr{T}_y)$.  Moreover, the
canonical inclusion $\sH(\Bell)\to\sH(\mathbf{m})$ induces an
inclusion $\HH^*_{\mathscr{T}_y}(\sH(\mathbf{m},y))\to
\HH^*_{\mathscr{T}_y}(\sH(\Bell,y))$.  (See Tymoczko's
paper~\cite{tymoczko} for results on localization applied to Hessenberg
varieties.)  The modules $\HH^*_{\mathscr{T}_y}(\sH(\mathbf{m},y))$ glue
together to form a local system $\mathscr{L}(\mathbf{m})$ over $\grs$ of
$\mathscr{A}$-modules.  This can be seen explicitly using 
Tymoczko's description of the equivariant cohomology of Hessenberg varieties
in terms of moment graphs. 

\begin{proposition}\label{pgkm}  Write $\pi_{\mathbf{m}}:\sHrs(\mathbf{m})\to\grs$
for the projection morphism, and let 
$\mathscr{A}_+$ denote the sheaf of ideals in $\mathscr{A}$ generated
by the positive degree elements.   Then we have an isomorphism of sheaves
$$
\mathscr{L}(\mathbf{m})/\mathscr{L}(\mathbf{m})\mathscr{A}_+\to
R^*\pi_{\mathbf{m}*}\mathbb{C}.
$$  
\end{proposition}
\begin{proof}
This follows from the fact that Hessenberg varieties are GKM spaces.
(See~\cite[\S2 and Proposition 5.4]{tymoczko}).
\end{proof}

For each $\mathbf{m}$, localization induces an inclusion
$\mathscr{L}(\mathbf{m})\to \mathscr{L}(\Bell)$ of
$\mathscr{A}$-modules.  Write $L(\mathbf{m})$ for the fiber, $\HH^*_T(\sH(\mathbf{m},y_0))$ of
$\mathscr{L}(\mathbf{m})$ over $y_0$.  Then $L(\mathbf{m})$ is free as an
$A$-module, and both $A$ and $L(\mathbf{m})$ are equipped with
compatible actions of $\pi_1(\grs,y_0)$.  If we write $A_+$ for the
ideal of positive degree polynomials, then we have
\begin{equation}
  \HH^*(\sH(\mathbf{m},y_0))=L/A_+L(\mathbf{m}),
\end{equation}
and the monodromy action of $\pi_1(\grs,y_0)$ on both sides is compatible. 

\begin{proposition}\label{snm}
  The action of $\pi_1(\grs,y_0)$ on $L(\mathbf{m})$ factors through
the homomorphism $\psi:\pi_1(\grs,y_0)\twoheadrightarrow S_n$.  
\end{proposition}
\begin{proof}
  The pullback $\mathscr{T}_Z$ of $\mathscr{T}$ to the $S_n$-cover
  $Z\to\grs$ is a constant group scheme, and the pullback of
  $\sHrs(\Bell)=Z$ to $Z$ is disjoint union of copies of $Z$
  indexed by elements of $S_n$.  It follows that the action of
  $\pi_1(\grs,y_0)$ on $\HH^*_T(\sH(\Bell,y_0))$ is trivial on the
  image of the map $\pi_1(Z,z_0)\to \pi_1(\grs,y_0)$.  In other words,
  the action of $\pi_1(\grs,y_0)$ on $\HH^*_T(\sH(\Bell,y_0))$
  factors through $S_n$.

Since $\mathscr{L}(\mathbf{m})\to\mathscr{L}(\Bell)$ is an inclusion of local systems,
we have a $\pi_1(\grs,y_0)$-equivariant inclusion $L(\mathbf{m})\to L(\Bell)$.  The result follows. 
\end{proof}

\begin{corollary}\label{cmac}
  The action of $\pi_1(\grs,y_0)$ on $\HH^*(\sHrs(\mathbf{m},y_0))$ 
induced from the local system $R^*\pi_{\mathbf{m}*}\mathbb{C}$
factors through $\psi:\pi_1(\grs,y_0)\twoheadrightarrow S_n$. 
\end{corollary}

\begin{proof}
  This follows directly from Propositions~\ref{snm} and~\ref{pgkm}.
\end{proof}

\begin{definition}
  The action of $S_n$ on $L(\mathbf{m})$ (resp.\ $\HH^*(\sHrs(\mathbf{m},y_0))$) coming from
  Proposition~\ref{snm} (resp.\ Corollary~\ref{cmac}) is called the \emph{monodromy action of~$S_n$}.
\end{definition}

\subsection{Monodromy action for $\sHrs(\Bell)$}

To make the monodromy action of $S_n$ on $L(\Bell)$ explicit, recall from
\S\ref{fcbt} that $z_0$ denotes the element of $Z_0:=\sH(\Bell,y_0)$
corresponding to $y_0$ with the standard ordering of its eigenspaces.
So $z_0=([e_1],\ldots, [e_n],y_0)$.  Given $\sigma\in S_n$, we have
$\sigma z_0=([e_{\sigma^{-1}1}],\ldots, [e_{\sigma^{-1}(n)}]; y_0)$.  The
cohomology group $\HH^*Z_{0}=\HH^0Z_{0}$ is simply the group of functions
$f:Z_{0}\to\mathbb{C}$.  If, for $w\in S_n$, we write $\delta_w$ for the function taking
$wz_0$ to $1$ and all other elements of $Z_0$ to $0$, then we have
$(\sigma \delta_w)(z)= \delta_w(\sigma^{-1}z)$.  From this it easily
follows that $\sigma\delta_w= \delta_{\sigma w}$.  

\begin{lemma}\label{ldot0}
As an $A$-module, $L(\Bell)$ is isomorphic to the module $A^{|S_n|}$
of functions from the set $S_n$ to $A$.  The monodromy action of $S_n$
on $L(\Bell)$ is given by 
$$
((w p)(v))(t_1,\ldots, t_n)=(p(w^{-1}v))(t_{w(1)},\ldots, t_{w(n)})
$$
where $v,w\in S_n$, $p\in A^{|S_n|}$ and $t_1,\ldots ,t_n$ are variables.
\end{lemma}
\begin{proof}
Under the identification $S_n\to Z_{0}$ given by $w\mapsto wz_0$,
the $\delta_w$ form a $\mathbb{C}$-basis of $\HH^0(Z_{0})$.  Moreover,
the map $\HH^0_T(Z_{0})\to \HH^0(Z_{0})$ is an $S_n$-equivariant isomorphism, and, 
under this identification, the $\delta_w$ freely generate $\HH^*_T(Z_{0})$
as an $A$-module.  The result now follows by direct verification using 
the fact that $S_n$ acts on $A$ as in \eqref{fcbtf}.
\end{proof}

\begin{corollary}\label{dotl}
  The monodromy action of $S_n$ agrees with Tymoczko's dot action
of $S_n$ on $\HH^*_T(\sH(\Bell,y_0))$. 
\end{corollary}
\begin{proof}
  This follows immediately by comparing the description of the monodromy action
in Lemma~\ref{ldot0} with Tymoczko's description of the dot action~\cite[\S 3.1]{tymoczko}.
\end{proof}

\begin{theorem}\label{dotm}
  Let $\mathbf{m}$ be a Hessenberg function.
The monodromy action of $S_n$ on $\HH^*_T(\sH(\mathbf{m},y_0))$ is the same as Tymoczko's dot action.
\end{theorem}

\begin{proof}
 Under Tymoczko's dot action, $\HH^*_T(\sH(\mathbf{m},y_0))$ is an $S_n$-equivariant $A$-submodule of $\HH^*_T(\sH(\Bell,y_0))$.  The same is true of the 
monodromy action of $S_n$.  Therefore, by Corollary~\ref{dotl}, the two actions must coincide.  
\end{proof}

\begin{corollary}\label{tdot}
  Tymoczko's dot action of $S_n$ on the non-equivariant cohomology group 
$\HH^*(\sH(\mathbf{m},y_0))$ coincides with the monodromy action.  
\end{corollary}
\begin{proof}
  Tymoczko defines the dot action on $\HH^*(\sH(\mathbf{m},y_0))$ as the
  dot action on the quotient $L(\mathbf{m})/A_+L(\mathbf{m})$.  The
  monodromy action is also given by this quotient.
\end{proof}

Since Tymoczko's dot action and the monodromy action of $S_n$ coincide, we 
will not distinguish between them from now on: it will be the only action 
of $S_n$ appearing in the remainder of the paper.

\begin{theorem}\label{tiy}
Let $s\in\gr$ be a regular element of type $\lambda$ and let $\pi=\pi_{\mathbf{m}}:\sH(\mathbf{m})\to\mathfrak{g}$.   Let $B(s)$ be a sufficiently small ball in $\mathfrak{g}$ centered at $s$.    Then, for each $k\in\mathbb{Z}$, there is a $\mathbb{C}$-vector space isomorphism
$$
\HH^0(B(s)\cap \grs,R^k\pi_*\mathbb{C})\cong
\HH^k(\sH(\mathbf{m},y_0))^{S_{\lambda}}.$$
\end{theorem}

\begin{proof}
By~\eqref{fgi} applied with $\mathcal{L}=R^k\pi_*\mathbb{C}$, we have $$\HH^0(B(s)\cap \grs,R^k\pi_*\mathbb{C})=
\HH^k(\sH(\mathbf{m},b))^{\pi_1(B(s)\cap\grs,b)}$$
where $b$ is any point in $B(s)\cap\grs$.   
The last vector space is isomorphic to the invariants of $\HH^k(\sH(\mathbf{m},y))$
under the local monodromy at $s$.  
The result then follows from Proposition~\ref{ryoung}.
\end{proof}

\begin{theorem}\label{dsf}
Suppose $s\in\mathfrak{g}^r$ is a regular element of type $\lambda$. 
Then, for each $k\in\mathbb{Z}$,
\begin{equation}
\label{eq:dsf}
\dim \HH^k(\sH(\mathbf{m},s))=\dim \HH^k(\sH(\mathbf{m},y_0))^{S_{\lambda}}.
\end{equation}
\end{theorem}
\begin{proof}
We are going to apply Theorem~\ref{t1} to the morphism $\pi:
\sH(\bm)\to\mathfrak{g}$. 
 Both the source and the target of $\pi$ are
  smooth, quasi-projective varieties.  Moreover, $\pi$ has relative
  dimension $|\bm|$.  (One way to check this is to use the fact that the
  projection $\pr_1:\sH(\bm)\to \mathscr{X}$ has relative dimension
  $\sum_{i=1}^n m_i$, while $\dim\mathscr{X}=\sum_{i=1}^{n-1} i$.
  Another way to see it, is to use the fact that the regular
  semisimple Hessenberg varieties have dimension $|\bm|$.)

By Corollary~\ref{pbetti}, we have
$$\dim \HH^i(\sH(\bm,s),\mathbb{C})=
\dim \HH^{2|\bm|-i}(\sH(\bm,s),\mathbb{C})$$ for all $i$.  

It follows then from Theorem~\ref{t1} that the local invariant cycle
map $$\HH^i(\sH(\bm,s))\to \HH^0(B(s)\cap\grs,R^i\pi_*\mathbb{C})$$ is an 
isomorphism, where $B(s)$ is any sufficiently small ball centered at
$s$ in $\mathfrak{g}$.   The result now follows from Theorem~\ref{tiy}. 
\end{proof}

Finally we can put all the pieces together to prove
Conjecture~\ref{conj:sw}.

\begin{theorem}
\label{thm:sw}
If $\chi_{\bm,d}$ denotes the dot action
on the cohomology group $\HH^{2d}$ of the regular semisimple
Hessenberg variety $\sH(\bm, s)$,
then $\ch \chi_{\bm,d}$ equals
the coefficient of $t^d$ in
$\omega X_{G(\bm)}(t)$.
\end{theorem}

\begin{proof}
By Theorem~\ref{thm:betti},
the left-hand side of Equation~(\ref{eq:dsf})
(in Theorem~\ref{dsf}) equals $c_{d,\lambda}(\bm)$
when $k=2d$.
On the other hand,
by Proposition~\ref{prop:frobenius},
the right-hand side of Equation~(\ref{eq:dsf})
equals the coefficient of $m_\lambda$ in $\ch \chi_{\bm,d}$.
\end{proof}

\bibliographystyle{plain}
\def\cprime{$'$}

\end{document}